\documentclass[a4paper,11pt]{article}
\usepackage{mathrsfs}
\usepackage{amsmath}                    
\usepackage{amssymb}                    
\usepackage{amsthm}                     
\usepackage{amsfonts}                   
\usepackage{color}
\usepackage{graphics,graphicx}
\usepackage{enumerate}
\usepackage{epsfig}
\usepackage{subcaption}

\usepackage{multirow}
\usepackage[english]{algorithm}
\usepackage{dsfont}

\allowdisplaybreaks

\newtheorem{theorem}{Theorem}[section]

\newtheorem{proposition}[theorem]{Proposition}
\newtheorem{lemma}[theorem]{Lemma}
\newtheorem{remark}[theorem]{Remark}

\newtheorem{ass}[theorem]{Assumption}

\numberwithin{equation}{section}
\def\ts{\thinspace}

\renewcommand{\epsilon}{\varepsilon}
\newcommand{\eps}{\varepsilon}

\newcommand{\R}{\mathbb{R}}
\newcommand{\bigo}{\mathcal{O}}
\newcommand{\C}{\mathbb{C}}
\newcommand{\N}{\mathbb{N}}

\newcommand{\T}{\mathbb{T}}
\newcommand{\Z}{\mathbb{Z}}
\newcommand{\K}{\mathcal{K}}
\newcommand{\U}{\mathcal{U}}

\def \s{\rm s}

\def\CC{{\mathbb C}}
\def\RR{{\mathbb R}}

\def\ds{\displaystyle}

\def\bs{\bigskip}

\def\pa{\partial}

\DeclareMathOperator{\IM}{Im}

\title{
A new class of uniformly accurate numerical schemes for highly oscillatory evolution equations
}

\author{ 
Philippe Chartier\textsuperscript{1}, Mohammed Lemou\textsuperscript{2}, Florian M\'ehats\textsuperscript{3} and Gilles Vilmart\textsuperscript{4}
}

\begin{document}
\footnotetext[1]{
INRIA Rennes, IRMAR and ENS Rennes, Campus de Beaulieu, F-35170 Bruz, France.
Philippe.Chartier@inria.fr}
\footnotetext[2]{
CNRS, IRMAR and ENS Rennes, Campus de Beaulieu, F-35170 Bruz, France.
Mohammed.Lemou@univ-rennes1.fr}
\footnotetext[3]{
Univ Rennes, INRIA, CNRS, IRMAR - UMR 6625, F-35000 Rennes, France.
Florian.Mehats@univ-rennes1.fr}
\footnotetext[4]{
Universit\'e de Gen\`eve, Section de math\'ematiques, 2-4 rue du Li\`evre, CP 64, CH-1211 Gen\`eve 4, Switzerland, Gilles.Vilmart@unige.ch}

\maketitle

\begin{abstract}
We introduce a new methodology to design {\em uniformly accurate methods} for oscillatory evolution equations. The targeted models are envisaged  in a wide spectrum of regimes, from non-stiff to highly-oscillatory. Thanks to an averaging transformation, the stiffness of the problem is softened, allowing for standard schemes to retain their usual orders of convergence. Overall,  high-order numerical approximations are obtained with errors and at a cost {\em independent} of the regime.

\smallskip

 \end{abstract}
\begin{center}
Communicated by Christian Lubich
\end{center}
\section{Introduction}
In this article, we consider time-dependent differential equations with (possibly high-) oscillations of the form
\begin{eqnarray} \label{eq:HOP}
\dot u^\eps(t) = f_{t/\eps}\left(u^\eps(t)\right)  \in X,  \quad u^\eps(0)=u_0 \in X, \quad t \in [0,1], 
\end{eqnarray}
where the function $(\theta,u) \in \T \times X \mapsto f_\theta(u)$ is assumed to be $1$-periodic w.r.t.\,$\theta$ on the torus $\T = \R / \Z \equiv  [0,1]$, and where $X$ denotes a Banach space endowed with a norm $\|\cdot\|_X$. The set $X$ may be simply $\R^d$, in which case (\ref{eq:HOP}) is a simple ordinary differential equation in finite dimension, or it  may be a {\em functional space}, 
say for instance the space of square-integrable functions $L^2(\R)$, in applications to {\em partial differential equations} such as the Schr\"odinger or the Klein-Gordon equations. Here, $\eps$ is a parameter whose value can vary in an  
interval of the form $]0,1]$. We emphasize that we do not restrict (\ref{eq:HOP}) to its asymptotic regime where $\eps$  tends to zero, but also allow $\eps$ to have a non-vanishing value, say close to $1$. 
\begin{remark}
The autonomous problem
\begin{equation} \label{eq:AP}
\dot v^\eps=\frac{1}{\eps}Av^\eps+f(v^\eps) \in X, \quad v^\eps(0)=v_0 \in X, \quad t \in [0,1], 
\end{equation}
where the map $\theta \mapsto e^{\theta A}$ is assumed to be $1$-periodic, can be straightforwardly rewritten within the format (\ref{eq:HOP}). As a matter of fact, the change of variable $v^\eps=e^{-\frac{t}{\eps} A}u^\eps$ transforms (\ref{eq:AP}) into  (\ref{eq:HOP}) with $f_\theta(u)=e^{-\theta A}f\left(e^{\theta A}u\right)$.
\end{remark}
Methods from the (recent and not so recent) literature are suitable either for the highly-oscillatory regime where $\eps$ is small, or for the non-stiff one where $\eps$ is close to $1$. However, they usually fail for certain values of $\eps$ in the range $]0,1]$,  indeed becoming  inaccurate or inefficient. For instance, the authors of this article have introduced several schemes
\cite{ccmss1,cmmv,cms,cmss3,cmss2}, inspired by the theory of {\em averaging} or the theory of {\em modified equations}, which happen to be efficient  for small values of $\eps$ and quite inaccurate for larger ones. This degraded behaviour appears as a consequence of the use of truncated series with no consideration of induced error terms. Recently, a new class of {\em uniformly accurate methods}, namely {\em two-scale methods} \cite{cclm,clm,clm2}, has emerged (see also the methods developed in \cite{bao-dong,faou-schratz2} using a different approach) that are capable of solving (\ref{eq:HOP}) with an accuracy and at a cost both independent of the value of $\eps \in ]0,1]$. Such methods rely on an explicit separation of scales $u^\eps(t) = \left. U^\eps(t,\theta)\right|_{\theta=t/\eps}$ with $\theta \in \T$ and the resolution of the transport equation 
\begin{align} \label{eq:transport}
\partial_t U^\eps(t,\theta) + \frac{1}{\eps} \partial_\theta U^\eps(t,\theta) = f_\theta(U^\eps(t,\theta))
\end{align}
with a well-chosen initial condition $U^\eps(0,\theta)=U_0^\eps(\theta)$. A numerical approximation $u_k^\eps$ of the solution $u^\eps(t_k)$ of (\ref{eq:HOP}) at time $t_k$ is then obtained from a $p^{th}$-order approximation of $U^\eps(t_k,\theta_k)$  by taking $\theta_k=t_k/\eps$ (see \cite{cclm} for the details of the technique). At a given computational cost $\chi \approx \frac{1}{\Delta t}$, this yields approximations on an uniform grid $t_k=k \Delta t$, $k=0,\ldots,K$, that satisfy the following error estimate 
\begin{align} \label{eq:error}
\max_{0 < \eps \leq 1} \; \max_{0 \leq k \Delta t \leq 1} \|u_k^\eps - u^\eps(k \Delta t)\|_{X} \leq C \Delta t^p,
\end{align}
where the constant $C$ may depend on $p$ but is {\em independent of $\Delta t$}  and $\eps$. Although the error behaviour (\ref{eq:error}) is unarguably favourable, the resolution of the transport equation (\ref{eq:transport}) as introduced in \cite{cclm,clm} requires to design specific schemes, a task that is not straightforward for high orders. This aspect of two-scale methods has left them open to criticisms and has motivated the search for a technique dealing with (\ref{eq:HOP}) in a more direct and more affordable way. 

In this work, we adopt the following strategy: through a suitable time-dependent (and periodic) change of variable, we {\bf transform equation (\ref{eq:HOP}) into a {\em smooth} differential equation which is then amenable to a standard numerical treatment}. By standard treatment, we mean here that {\bf any non-stiff numerical method (i.e. any integrator which could be applied to eq. \eqref{eq:HOP} with  $\eps=1$)  can be used for all values of $\eps \in ]0,1]$} and will exhibit its usual order of convergence when applied to the transformed equation. Unsurprisingly, the aforementioned change of variable is a key-ingredient of the theory of averaging. However, and this constitutes the main advantage of this strategy over the methods discussed above, no abrupt truncation is applied in the process, so that all the information originally contained in (\ref{eq:HOP}) is retained. In this way, uniformly accurate schemes are elaborated simply from {\em standard} numerical schemes by adjoining a discretization of the change of variable. 

In Section \ref{sect:aver}, we discuss the main ingredients of averaging theory. Firstly, we introduce in Subsection \ref{subsect:nutschell} the periodic change of variable and the averaged vector field for both stroboscopic and  standard averaging.  In particular, we write down, (i) the equation (\ref{eq:ideal}) that they {\em formally} satisfy, (ii) the inferred iteration (\ref{eq:phiiter}) from which they can be recursively built and, (iii) the corresponding defect (\ref{eq:defect}). Subsection \ref{subsect:standard} completes this expository section with the regularity estimates obtained so far in previous works. 

Section \ref{sect:tua} is the core of our paper. Subsection \ref{subsect:uam} presents the two methods that are later on analysed, namely the {\em micro-macro method} and the {\em pullback method}, together with possible variants. In Subsection \ref{subsect:bounds}, the defect and its time-derivatives after $n$ iterations are upper bounded by powers of $\eps$. Finally, we show in Subsection \ref{subsect:regularity} that the transformed equations obtained both by the micro-macro and pullback methods are smooth and indeed can be solved by non-stiff standard methods. 

Section \ref{sect:construction}  describes the actual implementation of the two most interesting methods, that is to say on the one hand, the {\em micro-macro method with standard averaging}, and on the other hand, the {\em pullback method with stroboscopic averaging}. Both methods are uniformly accurate. The first one appears to be easier to implement and cheaper. The second one has the advantage of being {\em geometric}, allowing in particular for the conservation of first integrals.

Finally, Section \ref{sect:numerical}  presents numerical experiments with the two techniques of Section \ref{sect:construction} used in combination with standard methods of orders $2$, $3$ and $4$. The H\'enon-Heiles problem and the Klein-Gordon equation are used as test-cases. As expected, the stroboscopic pullback technique exhibits a clear-cut preservation of invariants. In practice, all methods are evidently uniformly accurate. Up to our knowledge, this section demonstrates  the first example of geometric, high-order and uniformly accurate methods available in the literature. 

\section{Numerical methods pertaining to averaging theory} \label{sect:aver}
In this section, we recall the main result of averaging. Since our technique is intented for a large range of $\eps$-values, the following assumption on (\ref{eq:HOP}) appears as a natural prerequisite.
\begin{ass} \label{ass:K}
There exists an open set $\K \subset X$ with bounded closure $\bar \K$ and a open set $\U^0 \subset \K$, such that for all initial condition $u_0$ in $\U^{0}$ and all $\eps \in ]0,1]$, the solution $u^\eps(t)$ of 
\eqref{eq:HOP} 
exists and remains in $\K$ for all times $t \in [0,1]$. 
\end{ass}

\subsection{Averaging in a nutschell}
\label{subsect:nutschell}
We have chosen here to follow closely the presentation in \cite{ccmm} as it appears to be more convenient for the forthcoming developments. Other presentations can be found in classical references, e.g. \cite{perko} and \cite{SV}, or in recent papers, with a more algebraic point of view \cite{cms,cmss3,cmss2}. 
Generally speaking, the purpose of averaging methods is to  write the exact solution of (\ref{eq:HOP}) as the composition
\begin{align} \label{eq:phipsi}
u^\eps(t) = \Phi^\eps_{t/\eps} \circ \Psi^\eps_t \circ (\Phi^\eps_{0})^{-1}(u_0)
\end{align}
of a periodic {\em change of variable} $(\theta,u) \in \T \times \K \mapsto \Phi^\eps_\theta(u)$ and of the {\em flow map} $(t,u) \in [0,1] \times \K \mapsto \Psi^\eps_t(u)$ of an autonomous differential equation of the form 
\begin{align} \label{eq:psi}
\dot \Psi^\eps_t(u) = F^\eps(\Psi^\eps_t(u)), \quad \Psi^\eps_0(u)=u,
\end{align}
where the map $u \in \K \mapsto F^\eps(u)$ is a smooth vector field. The rationale behind averaging is apparent: the dynamics is sought after as the superposition of a {\em slow drift}, captured by $\Psi^\eps_t$, and of  {\em high oscillations}, captured by $\Phi^\eps_{t/\eps}$. Various design choices are possible for the  map $\Phi^\eps$ and we refer to \cite{cmss3} for a thorough investigation of their implications. 
In this paper, we shall impose that  either $\Phi^\eps_0$ or the average $\langle \Phi^\eps \rangle$ coincides with the {\em identity map} ${\rm id}_X$, where 
$$
\langle  g\rangle(u) = \int_0^1 g_\theta(u) d \theta
$$ 
stands for the average
on $\T$ of any continuous function $g:\T \times X\rightarrow X$, two procedures referred to respectively as {\em stroboscopic averaging} and {\em standard averaging}:
\begin{itemize}
\item \textit{Stroboscopic averaging $\Phi^\eps_0={\rm id}_X$}. Owing to periodicity, $\Phi^\eps_{t/\eps}={\rm id}_X$ for times $t\in \eps\N$ that are integer multiples of the period, hence the name ``stroboscopic''. When the problem \eqref{eq:HOP} under consideration embeds a structure (e.g. Hamiltonian structure) or possesses invariants (e.g. quadratic first integrals), stroboscopic averaging, being {\em geometric}, is to be preferred. 
\item \textit{Standard averaging $\langle \Phi^\eps \rangle={\rm id}_X$}.  This choice is more widely used and leads to a more simple implementation. In the context of kinetic theory, it is related to Chapman-Enskog method (see for instance in \cite{clm2}). 
\end{itemize}

For the decomposition in (\ref{eq:phipsi}) to exist, the maps $\Phi^\eps$ and $F^\eps$ would be required to satisfy the equation
\begin{align} \label{eq:ideal}
\frac{1}{\eps}\partial_\theta \Phi^{\eps}_\theta(u) +
 \partial_u \Phi^{\eps}_\theta(u) \; F^\eps(u) =  f_\theta \circ \Phi^{\eps}_\theta(u) \; \mbox{ with } \; F^{\eps}= \langle  \partial_u \Phi^{\eps} \rangle^{-1} \langle f \circ \Phi^{\eps} \rangle,
\end{align}
obtained by separation of scales.  
Unfortunately, it is a known fact that equation (\ref{eq:ideal}) has no rigorous solution:  only a {\em formal} solution  under the form of a divergent\footnote{Note, however, that this series is convergent for a linear right-hand side in (\ref{eq:HOP}) with bounded operator and for all sufficiently small~$\eps$.} series in $\eps$ can be exhibited. Truncation thereof generates errors which can be made polynomially or exponentially small w.r.t.\,$\eps$, depending on the smoothness of the vector field $f$ in (\ref{eq:HOP}) w.r.t.\,$u$. In averaging theory, as presented in \cite{ccmm}, approximations $\Phi^{[n]}$ and $F^{[n]}$ of $\Phi^\eps$ and $F^\eps$ are thus constructed from equation (\ref{eq:ideal}), starting from $\Phi^{[0]}={\rm id}_X$ and $F^{[0]}=\langle f  \rangle$, as follows 
\begin{eqnarray} \label{eq:phiiter}
\Phi^{[k+1]}_\theta = {\rm id}_X - \eps G^{[k+1]} +  \eps \int_0^\theta \hskip-1ex \Big(f_\tau \circ \Phi^{[k]}_\tau - \partial_u \Phi^{[k]}_\tau F^{[k]} \Big) d\tau,\quad  F^{[k]}= \langle  \partial_u \Phi^{[k]} \rangle^{-1} \langle f \circ \Phi^{[k]} \rangle,
\end{eqnarray}
for $k=0, \ldots,n-1$.
The functions $G^{[k+1]}$ could in principle be chosen arbitrarily: they vanish for {\em stroboscopic averaging}, or are determined to ensure that $\langle \Phi^{[k]} \rangle = {\rm id}_X$,  i.e.
\begin{align} \label{eq:G}
G^{[k+1]}=\left \langle\int_0^\theta \Big(f_\tau \circ \Phi^{[k]}_\tau - (\partial_u \Phi^{[k]}_\tau) \langle f \circ \Phi^{[k]} \rangle  \Big) d\tau \right \rangle,
\end{align} 
for {\em standard averaging}. 
Since equation (\ref{eq:ideal}) is impossible to satisfy exactly, a natural measure of success of the averaging idea is given by the size of the defect, defined by 
\begin{align} \label{eq:defect}
\delta_\theta^{[k]}(u) = \frac{1}{\eps} \partial_\theta \Phi^{[k]}_\theta(u) +
\partial_u \Phi^{[k]}_\theta(u) \; F^{[k]}(u) - f_\theta \circ \Phi^{[k]}_\theta(u), \quad k=0, \ldots,n.
\end{align}
\begin{remark} \label{rem:rel}
Starting from the stroboscopic decomposition $u^\eps(t) = \Phi^{\eps}_{t/\eps} \circ \Psi^\eps_t(u_0)$ with $\Phi^\eps_0={\rm id}_X$, we may write (once again formally at this stage)
$$
u^\eps(t) = \underbrace{\Phi^{\eps}_{t/\eps} \circ \langle \Phi^{\eps} \rangle^{-1}}_{\tilde \Phi^{\eps}_{t/\eps}} \circ \underbrace{\langle \Phi^{\eps} \rangle \circ \Psi^\eps_t \circ \langle \Phi^{\eps} \rangle^{-1}}_{\tilde \Psi^{\eps}_{t}} \circ \underbrace{\langle \Phi^{\eps} \rangle}_{ (\tilde \Phi^{\eps}_0)^{-1}} (u_0).
$$
It can indeed be checked, on the one hand, that 
$$
\langle \tilde \Phi^{\eps} \rangle = \langle  \Phi^{\eps} \circ \langle \Phi^{\eps} \rangle^{-1} \rangle  = \langle  \Phi^{\eps} \rangle \circ \langle \Phi^{\eps} \rangle^{-1} = {\rm id}_X \mbox{ and } \tilde \Phi^{\eps}_0 = \Phi^{\eps}_0 \circ \langle \Phi^{\eps} \rangle^{-1} = {\rm id}_X \circ \langle \Phi^{\eps} \rangle^{-1} = \langle \Phi^{\eps} \rangle^{-1},
$$
and on the other hand, that 
$$
\frac{d}{dt} \tilde \Psi^{\eps}_t(u_0) = \Big(\left(\partial_u \langle \Phi^{\eps} \rangle  \,  F \right) \circ \langle \Phi^{\eps} \rangle^{-1} \Big) \circ \tilde \Psi^{\eps}_t(u_0) =\langle f \circ \tilde \Phi^{\eps} \rangle \circ \tilde \Psi^{\eps}_t(u_0)= \tilde F^\eps \circ \tilde \Psi^{\eps}_t(u_0).
$$
This mean that the maps $\tilde \Phi^{\eps}$, $\tilde \Psi^{\eps}$ and $\tilde F^\eps$ are those corresponding to standard averaging and that we have the following relation on the defects, 
$$
\tilde \delta_\theta^{\eps} = \delta_\theta^{\eps} \circ \langle \Phi^\eps \rangle^{-1}.
$$ 
Conversely, we may relate $\Phi^{\eps}$ to $\tilde \Phi^{\eps}$ through the formula
$$
\Phi^{\eps}_\theta = \tilde \Phi^{\eps}_\theta \circ (\tilde \Phi^{\eps}_0)^{-1}.
$$
\end{remark}
\subsection{Standard error estimates in averaging theory}
\label{subsect:standard}
The main result of \cite{ccmm} now ensures that the maps 
$\Phi^{[n]}$ and $F^{[n]}$ given by iteration (\ref{eq:phiiter}) are well-defined and satisfy equation (\ref{eq:ideal}) up to small error terms of the form ${\cal O}(\eps^{n})$. It is derived under additional analyticity assumptions on $f$. As in \cite{ccmm}, we thus introduce the complexification $X_\C$ of $X$ (which is just $X_\C=\C^d$ in the case where  $X=\R^d$), defined as 
$$
X_\C=\{U:=u_r+i u_i, (u_r,u_i) \in X^2\},
$$
where $u_r=\Re(U) \in X$ and  $u_i = \Im(U) \in X$ are the real and imaginary parts of $U$.
The space $X_\C$ is still a Banach space when endowed with the norm (which we will simply denote $\|\cdot\|$ in the sequel)
$$
\|U\|_{X_\C} := \sup_{\lambda \in \C^*} \frac{ \|\Re(\lambda U)\|_X}{|\lambda|},
$$
and we can consider for $\rho \geq 0$ the closed set 
$$
{\cal K}_\rho =  \{u+\tilde u: u \in \bar \K, \tilde u \in X_\C,  \|\tilde  u\| \leq \rho \}.
$$
Now, given a bounded map $(\theta,u) \in \T \times {\cal K}_\rho \mapsto \varphi_\theta (u)$, we finally denote 
$$
\|\varphi\|_{\rho} := \sup_{\theta \in \T, \; u \in \K_\rho} \|\varphi_\theta(u)\|
$$ 
and if $\varphi$ has $p$ bounded derivatives w.r.t.\,$\theta$
$$
\|\varphi\|_{\rho,p} := \max_{\nu=0,\ldots,p} \|\partial^\nu_\theta \varphi_\theta\|_{\rho}.
$$
The main hypothesis that was required in \cite{ccmm} is the following. \\ \\
\begin{ass} \label{ass:A}
The function $(\theta,u) \mapsto f_\theta(u)$ is  continuous w.r.t.\,$\theta$ and there exists $R>0$ such that for all $\theta \in \T$, the function $f_\theta(\cdot)$ can be extended to an {\em analytic}  map from $\K_{2R}$ to $X_\C$. In addition, there exists $M > 0$ such that 
\begin{align}
\|f\|_{2R} \leq M.
\end{align}
\end{ass} 

An immediate consequence is that for all $u_0+\tilde u_0 \in {\cal U}^0_i$ where 
$$
{\cal U}^0_i  :=  \left\{u_0+\tilde u_0: u_0 \in {\cal U}^0, \tilde u_0 \in X_\C,  \|\tilde  u_0\| \leq R e^{-\frac{M}{R}} \right\},
$$
for all $\eps \in]0,1]$ and all $t \in [0,1]$, the solution $u^\eps(t)$ of 
\begin{align} \label{eq:gode}
\dot u^\eps(t) = f_{t/\eps}(u^\eps(t)), \quad u^\eps(0)=u_0 + \tilde u_0,
\end{align}
remains in $\K_R$. By Assumption \ref{ass:A}, $f_\theta$ is indeed Lipschitz on $\K_R$ with Lipschitz constant $M/R$, so that by virtue of the Gronwall lemma and Assumption \ref{ass:K}, $u^\eps(t)$ remains in $\K_R$ for $\|\tilde u_0\| \, e^{M/R} \leq R$.  Now, the following result has been proved in the case of stroboscopic averaging.
\begin{theorem} [See Castella, Chartier, Murua and M\'ehats \cite{ccmm}] \label{th:mta}
Suppose that Assumptions \ref{ass:K} and \ref{ass:A} are satisfied and let $\eps_0:=R/(8M)$. For every positive integer $n$ and positive $\eps$ such that $(n+1) \eps \leq \eps_0$, the maps $\Phi^{[n]}$ and $F^{[n]}$ given by iteration (\ref{eq:phiiter}) with all $G^{[k]} \equiv 0$, are well-defined, $\Phi^{[n]}$ is continuously differentiable w.r.t.\,$\theta$ and both $\Phi^{[n]}$ and its average $\langle \Phi^{[n]} \rangle$ are invertible with analytic inverse on $\K_R$. Moreover, the following estimates hold for all $0 < \eps < \eps_0$,
\begin{align*}
(i) \quad \|\Phi^{[n]} - {\rm id}\|_R \leq \frac{R}{2(n+1)} \quad  (ii) \quad \|F^{[n]}\|_R \leq 2 M \quad  (iii) \quad \|\delta^{[n]}\|_R \leq 2 M \left( 2 (n+1) \frac{\eps}{\eps_0}\right)^n.
\end{align*}
\end{theorem}
It is known that by choosing $(n+1)$ appropriately depending on $\eps$, the upper-bound of $\delta^{[n]}$ can be made exponentially small in $\eps$, but this is not our main interest here. 

\section{The uniform accuracy paradigm for averaging methods} \label{sect:tua}
The vector field $F^{[n]}$ appearing in Theorem \ref{th:mta} 
has a Taylor expansion with respect to $\eps\rightarrow 0$ which coincides to (within error terms of size $\eps^n$) with a formal series $\langle f \rangle  + \eps F_2 + \eps^2 F_3 + \ldots$, whose successive terms can be determined analytically \cite{ccmm,cmss3}. 
Hence,  it makes perfect sense for small values of $\eps$ to simply ignore the discrepancy in equation (\ref{eq:phipsi}) and to approximate \eqref{eq:HOP} from the differential  equation
$\dot u = F^{[n]}(u)$ (complemented with the computation of $\Phi^{[n]}$), yielding an error of size $\bigo(\eps^n)$.  However, if the value of $\eps$ is not so small -- close to $1$ for instance --, this procedure becomes completely inaccurate. The aim of the forthcoming techniques is  precisely to resolve this difficulty.

\subsection{Uniformly accurate averaging: micro-macro and pullback methods}
\label{subsect:uam}

A uniformly accurate method should thus satisfy two requirements: on the one hand,  it should filter high-oscillations in one way or another -- and thus rely on the principle of averaging -- in order to be {\em efficient} for small values of $\eps$, and on the other hand,  it should incorporate the relevant information for large values of $\eps$ so as to remain accurate. This leads to the two following alternative methods considered in this paper.

\paragraph{Micro-macro method.} If we insist  on solving an autonomous equation of the form (\ref{eq:psi}), i.e.
$$
\dot \Psi^{[n]}_t(u) = F^{[n]} \left( \Psi^{[n]}_t (u) \right), \quad \Psi^{[n]}_0 (u)=u,
$$
then it is of paramount importance to retain the information contained in the remainder 
$$
{\cal G}^{[n]}_t(u_0) = u^\eps(t) - \Phi^{[n]}_{t/\eps} \circ \Psi^{[n]}_t \circ (\Phi^{[n]}_{0})^{-1}( u_0),
$$
which obeys the following differential equation 
\begin{align*}
\frac{d}{dt} {\cal G}^{[n]}_t(u_0)  &= \dot u^\eps(t) - \Big( \frac{1}{\eps}\partial_\theta \Phi^{[n]}_{t/\eps}
  + \partial_u \Phi^{[n]}_{t/\eps}  \; F^{[n]}  \Big)(\Psi^{[n]}_t (\check u_0)) \\
&=f_{t/\eps} \Big(\Phi^{[n]}_{t/\eps} \circ \Psi^{[n]}_t(\check u_0)+{\cal G}^{[n]}_t(u_0)\Big) - f_{t/\eps} \Big(\Phi^{[n]}_{t/\eps} \circ \Psi^{[n]}_t(\check u_0)\Big) -  \delta_{t/\eps}^{[n]}(\Psi^{[n]}_t(\check u_0))
\end{align*}
with $\check u_0= (\Phi^{[n]}_{0})^{-1}(u_0)$. 
Upon using the Gronwall lemma, it can be shown (see Section \ref{subsect:regularity}) that  ${\cal G}^{[n]}_t(u_0)$ and its $n$ first time-derivatives are uniformly bounded w.r.t.\,$\eps$ provided $\delta^{[n]}$ (see relation (\ref{eq:defect})) and its first $n$ derivatives with respect to $\theta$ are of size $\eps^{n}$. 
This leads to the following {\em micro-macro} system
\begin{eqnarray} \label{eq:micromacro}
\left\{
\begin{array}{llll}
\dot v &=& F^{[n]}(v), & \quad v(0) = (\Phi^{[n]}_0)^{-1}(u_0) \\ 
\dot w &=& f_{t/\eps} \Big(\Phi^{[n]}_{t/\eps} (v)+w \Big) - \Big( \frac{1}{\eps}\partial_\theta \Phi^{[n]}_{t/\eps}
  + \partial_u \Phi^{[n]}_{t/\eps}  \; F^{[n]}  \Big)(v), & \quad w(0)=0
  \end{array}
  \right.
 \end{eqnarray}
from which the solution of \eqref{eq:HOP} can be recovered as $u^\eps(t)=\Phi^{[n]}_{t/\eps}(v(t))+w(t)$. 
\paragraph{Pullback method.}
If we insist on satisfying an equation of the form (\ref{eq:phipsi}) exactly, i.e.
\begin{align} \label{eq:phipsin}
u^\eps(t) = \Phi^{[n]}_{t/\eps} \circ \Psi^{[n]}_{t} \circ (\Phi^{[n]}_{0})^{-1} (u_0), 
\end{align}
then an autonomous equation of the form (\ref{eq:psi}), i.e. 
$$
\dot \Psi_t^{[n]}(\check u_0) = F \circ \Psi_t^{[n]}(\check u_0), \quad \Psi_0^{[n]}(u_0) = \check u_0,
$$
where  $\check u_0 = (\Phi^{[n]}_{0})^{-1} (u_0)$, 
can not hold and should instead be amended to
\begin{align} \label{eq:psiR}
\dot \Psi_t^{[n]}(\check u_0) = F^{[n]}(\Psi_t^{[n]}(\check u_0)) + {\cal R}^{[n]}_{t/\eps}(\Psi_t^{[n]}(\check u_0)), \quad \Psi^{[n]}_0(\check u_0)=\check u_0.
\end{align}
As a matter of fact, by differentiating (\ref{eq:phipsin}) as follows (with $\check u_0 = (\Phi^{[n]}_{0})^{-1} (u_0)$)
\begin{align*}
0 &= \dot u^\eps(t) - \frac{1}{\eps}\partial_\theta \Phi^{[n]}_{t/\eps}  (\Psi^{[n]}_t (\check u_0))
  - \partial_u \Phi^{[n]}_{t/\eps} (\Psi^{[n]}_t (\check u_0)) \; \dot \Psi^{[n]}_t (\check u_0) \\
&= \partial_u \Phi^{[n]}_{t/\eps} (\Psi^{[n]}_t (\check u_0))\;  (F^{[n]}(\Psi^{[n]}_t (\check u_0))- \dot \Psi^{[n]}_t(\check u_0) )-  \delta_{t/\eps}^{[n]}(\Psi^{[n]}_t (\check u_0)),
\end{align*}
we obtain  (\ref{eq:psiR}) provided that 
\begin{equation} \label{eq:defRtheta}
{\cal R}^{[n]}_{\theta} = -\left(\partial_u \Phi^{[n]}_\theta \right)^{-1} \delta^{[n]}_\theta. 
\end{equation}
This leads to the following pulled back equation 
\begin{align} \label{eq:pullback}
\dot v = \left(\partial_u \Phi^{[n]}_{t/\eps}\right)^{-1} \Big(f_{t/\eps} \circ \Phi^{[n]}_{t/\eps} (v)-\frac{1}{\eps} \partial_\theta \Phi^{[n]}_{t/\eps}(v)\Big) , \quad v(0) = (\Phi^{[n]}_0)^{-1}(u_0),
\end{align}
from which the solution of \eqref{eq:HOP} can be recovered as $u^\eps(t)=\Phi^{[n]}_{t/\eps}(v(t))$. 

\bigskip
Both the micro-macro  \eqref{eq:micromacro} and the pullback \eqref{eq:pullback} formulations will later prove to be useful for the design of uniformly accurate numerical schemes. 
They can be considered for stroboscopic and standard averagings,
yielding four different approaches (see  Table \ref{tab:4methods}).
\begin{table}[h!]
\begin{center}
\begin{tabular}{|c|c|c|}
\hline 
 & Stroboscopic averaging  & Standard averaging \\
  \hline 
Pullback method  & $\Phi^{[n]}_0 ={\rm id}_X$, ${\cal G}^{[n]} \equiv 0$ & $\langle \Phi^{[n]} \rangle ={\rm id}_X$, ${\cal G}^{[n]} \equiv 0$ \\  \hline 
Micro-macro method & $\Phi^{[n]}_0 ={\rm id}_X$, ${\cal R}^{[n]} \equiv 0$ & $\langle \Phi^{[n]} \rangle ={\rm id}_X$, ${\cal R}^{[n]} \equiv 0$ \\   \hline 
\end{tabular}
\caption{Designing choices for uniformly accurate averaging. \label{tab:4methods}}
\end{center}
\end{table}

The numerical resolution of either system \eqref{eq:micromacro} or system \eqref{eq:pullback} by a standard non-stiff method now requires estimates of time-derivatives of their right-hand side. As discussed above, this boils down to getting estimates of $\partial^{k}_t \delta^{[n]}_{t/\eps}$ for $k=0,\ldots$, which is the objective of next subsection.

\subsection{Bounds on the defect $\delta^\eps$}
\label{subsect:bounds}
We introduce the two  non-linear operators, which, to a map $(\theta,u)  \in \T\times {\cal K}_\rho \mapsto \varphi_\theta(u)$ with invertible partial derivative $\partial_u \langle \varphi_\theta \rangle$ associate
the mappings $(\theta,u)  \in \T \times {\cal K}_\rho \mapsto \Lambda(\varphi)_\theta(u)$ and $(\theta,u)  \in \T\times K_\rho \mapsto \Gamma^\eps(\varphi)_\theta(u)$  defined as 
\begin{eqnarray}
\label{Lambda}
\Lambda (\varphi)_{\theta}(u) &=& f_\theta \circ \varphi_{\theta} (u) - \frac{\partial \varphi_{\theta}}{\partial u} (u)\left(\frac{\partial \langle\varphi\rangle}{\partial u} (u)\right)^{-1} \left\langle f \circ \varphi\right\rangle(u), \\
\label{Gamma}
\Gamma^\eps (\varphi)_{\theta}(u) &=& u + \eps  \left(\int_0^{\theta}  \Lambda (\varphi)_{\xi}(u)  \, d \xi - \gamma \left\langle \int_0^{\theta}  \Lambda (\varphi)_{\xi}(u)  \, d \xi  \right\rangle \right),
\end{eqnarray}
with $\gamma=0$ for stroboscopic averaging and $\gamma=1$ for standard averaging. In the sequel, we shall constantly need the following assumption (more stringent than our previous Assumption \ref{ass:A}).

\begin{ass} \label{ass:f}
The function $(\theta,u) \mapsto f_\theta(u)$ is $p$-times ($p \geq 1$) continuously differentiable w.r.t.\,$\theta$.
Moreover, there exists $R>0$ such that for all $\theta \in \T$, the function $u \mapsto f_\theta(u)$ and all its derivatives $u \mapsto \partial^\nu_\theta f_\theta(u)$ up to order $\nu=p$ can be extended to  {\em analytic}  maps from $\K_{2R}$ to $X_\C$ that are {\it uniformly} bounded w.r.t.\,$(\theta,u)\in \T\times \K_{2R}$. More precisely, there exists a constant $M > 0$ such that for all $0 < \sigma \leq 3$,
\begin{align}
\max_{\nu=0,\ldots,p} \frac{\sigma^\nu}{\nu!} \|\partial^\nu_\theta f\|_{2R} \leq M.
\end{align}
\end{ass}
Stating upper-bounds (polynomial and exponential in terms of $\eps$) of ${\cal G}^{[n]}$ or ${\cal R}^{[n]}$  now boils down to getting upper-bounds,  as well as regularity estimates, of the defect $\delta^{[n]}$. 
\begin{proposition} \label{prop:mta2}
Suppose that Assumptions \ref{ass:K} and \ref{ass:f} are satisfied and let $\eps_0:=R/(8M)$ and denote $r_n=\frac{R}{n+1}$ and $R_k=2R-kr_n$, $k=0,\ldots,n+1$. For all positive integer $n$ and positive $\eps$ such that $(n+1) \eps \leq \eps_0$, the maps $\Phi^{[k]}$, $k=0,\ldots,n$, defined by iteration (\ref{eq:phiiter}) are analytic on $\K_{R_k}$ and both $\Phi^{[n]}$ and its average $\langle \Phi^{[n]} \rangle$ are invertible with analytic inverse on $\K_R$.
Furthermore, all maps $\Phi^{[k]}$ are $(p+1)$-times continuously differentiable w.r.t.\,$\theta$ with derivatives analytic on $\K_{R_k}$ and upper-bounded:
\begin{eqnarray*}
\forall k=0,\ldots,n,  \quad \forall 1 \leq \nu \leq p+1, \qquad \|\partial^\nu_\theta \Phi^{[k]}\|_{R_k} \leq 8 \,  \eps \, M \, \nu!
\end{eqnarray*}
\end{proposition}
\begin{proof}
Owing to Remark \ref{rem:rel}, we can restrict ourselves to the case of stroboscopic averaging, for which the (rather technical) proof is given in Appendix section.
\end{proof}
\begin{theorem} \label{th:delta}
Suppose that Assumptions \ref{ass:K} and \ref{ass:f} are satisfied and let $\eps_0:=R/(8M)$. For all positive integer $n$ and positive $\eps$ such that $(n+1) \eps \leq \eps_0$, let the map $\Phi^{[n]}$ and $F^{[n]}$ be defined by iteration (\ref{eq:phiiter}). Then the defect $\delta^{[n]}$ defined by  (\ref{eq:defect}) is $p$-times differentiable and satisfies the following estimate:
\begin{eqnarray*}
\| \delta^{[n]}\|_{R,p} \leq  2 \, M \, \Big(2 \, {\cal Q}_p \, (n+1) \frac{\eps}{\eps_0}\Big)^{n},
\end{eqnarray*}
where ${\cal Q}_p  \leq 5+ 2^{p+1} e^{5/2} \, p!$.
\end{theorem}
For the proof of Theorem \ref{th:delta}, we need the following lemma with proof postponed in appendix.
\begin{lemma} \label{lem:basicdiff}
Let $0 < \delta < \rho \leq 2R$ and consider two periodic, near-identity mappings 
$(\theta,u) \in \T \times {\cal K}_\rho \mapsto \varphi_\theta(u)$ and $(\theta,u) \in \T \times {\cal K}_\rho \mapsto \tilde \varphi_\theta(u)$, analytic on $K_\rho$ and satisfying 
\begin{align} \label{eq:nearidentity}
\|\varphi-{\rm Id}\|_\rho \leq \frac{\delta}{2} \quad \mbox {and} \quad\|\tilde \varphi-{\rm Id}\|_\rho \leq \frac{\delta}{2}.
\end{align}
Let $\eps_0:=R/(8M)$ and suppose that they are $p$-times continuously differentiable w.r.t.\,$\theta$ and satisfy the estimates
$$
\forall \, 0 < \eps < \eps_0, \, \forall \, 1 \leq \nu \leq p, \quad \max\Big(\|\partial^\nu_\theta \varphi\|_\rho, \, \|\partial^\nu_\theta \tilde \varphi\|_\rho\Big) \leq \beta_\nu:= 8 \, M \, \eps \, \nu! 
$$
Then the following estimates hold true 
\begin{align*}
\forall \, 0 < \eps < \eps_0, \, \forall \, 0 \leq \nu \leq p, \qquad & \|\Lambda (\varphi) - \Lambda (\tilde \varphi)\|_{\rho-\delta,\nu} \leq \frac{\xi \, Q_\nu(\xi)}{8 \, \eps}  \|\varphi - \tilde \varphi\|_{\rho,\nu},
\\
 & \|\Gamma^\eps (\varphi) - \Gamma^\eps (\tilde \varphi)\|_{\rho-\delta,\nu+1}
\leq  \frac{\xi \, Q_\nu(\xi)}{8}  \|\varphi - \tilde \varphi\|_{\rho,\nu},
\end{align*}
where $Q_\nu(\xi) \leq 1+2\xi + 2^{\nu+1} \sqrt{e} \, e^{\xi} \, \nu!$ and $\xi=\frac{16 M \eps}{\delta}$. 
\end{lemma}
\begin{proof}[Proof of Theorem \ref{th:delta}]
Again, owing to Remark \ref{rem:rel}, we limit ourselves to the case of stroboscopic averaging. By virtue of Proposition \ref{prop:mta2}, the iterates $\Phi^{[k]}$, $k=0,\ldots,n$ are $p+1$ times differentiable and satisfy the assumptions of Lemma  \ref{lem:basicdiff}. 
By definition, we have for  $n \geq 1$
$$
\delta^{[n]}_{\theta}= \Lambda (\Phi^{[n-1]})_{\theta}-\Lambda (\Phi^{[n]})_{\theta}.
$$
Hence, using Lemma  \ref{lem:basicdiff} with $\delta=\frac{R}{n+1}$ and successively $\rho=R_n$, $\rho=R_{n-1}$, ..., $\rho=R_1$ yields with $\xi=\frac{16 M \eps (n+1)}{R}$:
\begin{align*}
\|\delta^{[n]}\|_{R,p}  
&\leq 
 \frac{1}{8 \eps} \xi Q_p(\xi)  \|\Phi^{[n]}-\Phi^{[n-1]}\|_{R_n,p} \\
& =
\frac{1}{8 \eps} \xi Q_p(\xi)  
\left\|\Gamma^\eps\left(\Phi^{[n-1]}\right)-\Gamma^\eps\left(\Phi^{[n-2]}\right)\right\|_{R_{n},p}
\\
&
\leq 
\frac{1}{\eps} \Big(\frac{\xi Q_p(\xi)}{8}\Big)^2
\left\|\Gamma^\eps\left(\Phi^{[n-2]}\right)
-\Gamma^\eps\left(\Phi^{[n-3]}\right)\right\|_{R_{n-1},p} \\
& \leq
\ldots 
\leq \frac{1}{\eps} \Big(\frac{\xi Q_p(\xi)}{8}\Big)^{n} \|\Phi^{[1]}-\Phi^{[0]}\|_{R_{1},p}.
\end{align*}
In addition, a direct estimation gives 
$
\|\Phi^{[1]}-\Phi^{[0]}\|_{R_1,p} \leq 2 \eps  \, M,
$
so that 
$$
\|\delta^{[n]}\|_{R,p} \leq 2 M \Big(\frac{\xi Q_p(\xi)}{8}\Big)^n.
$$
Now, $(n+1) \eps \leq \eps_0$, so that $\xi \leq 2$ and  we have 
$$
Q_{p}(\xi) \leq {\cal Q}_p :=  1+4+2^{p+1} \, p! e^{5/2} \leq 5 + 13 \, 2^{p+1} \, p! 
$$
\end{proof}
\subsection{Uniform accuracy of numerical schemes} 
\label{subsect:regularity}
We are now in position to prove the main result of Section \ref{sect:tua} which permits to apply a non-stiff standard integrator to the pullback and micro-macro formulations.
We consider the numerical solution by a standard scheme of the pullback problem \eqref{eq:pullback}
and of the micro-macro problem \eqref{eq:micromacro}.
\begin{theorem} \label{thm:regularity}
Suppose that Assumptions \ref{ass:K} and \ref{ass:f} are satisfied for a given $p\ge 1$ and let $\eps_0:=R/(8M)$. For problems \eqref{eq:pullback} and \eqref{eq:micromacro} with $0 < (n+1) \eps \leq \eps_0$, consider an approximation $(v_k)_{0 \leq k \leq N}$ of \eqref{eq:pullback} or $(v_k,w_k)_{0 \leq k \leq N}$ of \eqref{eq:micromacro} 
at times $t_0=0<t_1<\ldots<t_N=1$ obtained by a standard {\em stable} method of {\em nonstiff order $q$}, i.e. a method which exhibits order $q$ of convergence when applied to \eqref{eq:HOP} with $\eps=1$. Then, these approximations yield a uniformly accurate approximation of order $r=\min(p,q,n)$ of the solution of \eqref{eq:HOP}. Precisely, we have for the pullback method,
$$
\max_{0 \leq k \leq N} \|\Phi^{[n]}(v_k) - u^\eps(t_k)\|_X \leq C\Delta t^r,
$$
 and for the micro-macro method,
$$
\max_{0 \leq k \leq N} \|\Phi^{[n]}(v_k) + w_k - u^\eps(t_k)\|_X \leq C\Delta t^r,
$$
where the constant $C$ is {\em independent of  $\eps$} and the time grid size $\Delta t=\max_k (t_{k+1}-t_k)$ assumed small enough.
\end{theorem}
\begin{proof}
We first note that according to Proposition \ref{prop:mta2},
$\Phi^{[n]}$ is analytic and bounded on $\K_R$, hence it is Lipschitz continuous, so that all we have to prove is that the approximations $v_k$ and $(v_k,w_k)$ to equations \eqref{eq:pullback}
and  \eqref{eq:micromacro} are of order $r$. As far as the pullback method is concerned, equation \eqref{eq:pullback} can be written as (see (\ref{eq:psiR}))
$$
\dot v = F^{[n]}(v) + {\cal R}_{t/\eps}^{[n]}(v), \quad v(0) = (\Phi^{[n]}_0)^{-1}(u_0),
$$
where $F^{[n]}$ is Lipschitz and where ${\cal R}_{\theta}^{[n]}$ is defined in \eqref{eq:defRtheta}. The result thus follows from the existence of the derivatives of ${\cal R}_{\theta}^{[n]}$ w.r.t.\,$\theta$ up to order $p$ and the estimates
$$
\max_{1 \leq \nu + \xi \leq p} \|\partial_t^\nu \partial_u^\xi ({\cal R}^{[n]}_{t/\eps})\|_{R} \leq  C\eps^{n-\nu},
$$
which are a consequence of Theorem \ref{th:delta} and of the Leibniz formula.

%
Consider now the micro-macro system \eqref{eq:micromacro}.
By definition of  $\delta_\theta^{[n]}$, it may be expressed as 
\begin{align} \label{eq:ss}
\left\{
\begin{array}{ccll}
\dot v &=& F^{[n]}(v), & v(0) = (\Phi^{[n]}_0)^{-1}(u_0), \\
 \dot w &=& L^{[n]}\left(t/\eps,t,w\right) \, w + b^{[n]}(t/\eps,t),  &  w(0)=0,
\end{array}
\right.
\end{align}
with 
$$
L^{[n]}(\theta,t,w)=\left( \int_0^1 \partial_u f_{\theta}\Big(\Phi^{[n]}_{\theta} (v(t)) + \mu w \Big) d\mu \right), \qquad 
b^{[n]}(\theta,t)=-  \delta_{\theta}^{[n]}(v(t)), 
$$
and where we have written in \eqref{eq:micromacro}
\begin{align*}
f_{t/\eps} \Big(\Phi^{[n]}_{t/\eps} (v(t)) + w \Big) - f_{t/\eps} \Big(\Phi^{[n]}_{t/\eps} (v(t))\Big)  
 &= \left( \int_0^1 \partial_u f_{t/\eps}\Big(\Phi^{[n]}_{t/\eps} (v(t)) + \mu w \Big) d\mu \right) w.
 \end{align*}
Here, both the operator function
$L^{[n]}(\theta,t,w)$ and the source term $b^{[n]}(\theta,t)=-  \delta_{\theta}^{[n]}(v(t))$ are of class $C^p$ w.r.t.\,$\theta \in \T$, of class $C^\infty$ w.r.t.\,$t \in [0,T]$ and analytic w.r.t.\,$w$ on $\K_R$ with bounds
\begin{align*}
{\cal L}^{[n]} := \max_{0 \leq \nu+\mu+\xi \leq p} \; \sup_{(\theta,t) \in \T \times [0,T]} \|\partial_\theta^\nu \partial_t^\mu \partial_w^\xi L^{[n]}(\theta,t,\cdot)\|_R 
\end{align*}
and 
\begin{align*}
 {\cal B}^{[n]} := \max_{0 \leq \nu+\mu \leq p} \; \sup_{(\theta,t) \in \T \times [0,T]} \|\partial_\theta^\nu \partial_t^\mu b^{[n]}(\theta,t)\|_X \leq \mathfrak{b}^{[n]} \, \eps^n, 
\end{align*}
where $\mathfrak{b}^{[n]}$ is a constant independent of $\eps$. 
Exploiting the fact that $\langle b^{[n]}(\cdot,t)\rangle=0$  in the application of the Gronwall lemma allows to assert that there exists $M_0$ such that 
$$
\forall t \in [0,1], \quad \|w(t)\|_X \leq M_0  \, \mathfrak{b}^{[n]} \, \exp({\cal L}^{[n]} t) \,  \eps^{n+1}.
$$ 
Then, by differentiating \eqref{eq:ss} up to $\tilde r = \min(p,n)$ times and using again the Gronwall lemma, it can be verified that there exists a constant  $M_{\tilde r} \in \R^*$ such that 
$$
\forall \nu=0, \ldots,\tilde r+1, \quad \forall t \in [0,1], \quad \|\partial_t^\nu w(t)\|_X \leq M_{\tilde r}  \, \mathfrak{b}^{[n]} \, \exp( {\cal L}^{[n]} t) \,  \eps^{n+1-\nu}.
$$
As for the numerical approximations $w_k$ of $w(t_k)$, the boundedness of $L$  and the stability of the scheme allow to write 
$$
\|w_{k+1} \|_X \leq (1+ \alpha {\cal L}^{[n]} \Delta t) \|w_{k} \|_X + \beta \mathfrak{b}^{[n]} \eps^n \, \Delta t, \quad k=0, \ldots, N-1, 
$$
where $\alpha$ and $\beta$ are two constants depending on the scheme, so that 
$$
\forall k =0,\ldots, N, \quad \|w_{k} \|_X \leq \beta \, \mathfrak{b}^{[n]} \, \exp(\alpha {\cal L}^{[n]}) \, \eps^n
$$
When the numerical scheme advances the solution from $t_k$ to $t_{k+1}$, it uses the right-hand side of \eqref{eq:micromacro} evaluated at $w_k$. The fact that $w_k$ is of order $\eps^n$ implies that all derivatives of this right-hand side up to order $\tilde r+1$, when evaluated at $w_k$, are bounded in $\eps$. This allows to assert that the scheme retains its usual order $q$, provided $q \leq \tilde r$. 
\end{proof}

\section{Construction of uniformly accurate integrators}
\label{sect:construction}

In this section, we explain how our framework allows for the derivation of effective
uniformly accurate schemes. 

\subsection{Standard averaging micro-macro methods}
\label{subsection:mm}
We now detail the procedure for the construction of a scheme with accuracy of order $1$ uniformly with respect to $\eps\in]0,1]$. The simplest change of variable can be obtained after a single iteration in \eqref{eq:phiiter}, i.e. for $n=1$, which yields
\begin{equation}\label{eq:eulex}
\Phi^{[1]}_\theta(w) = w+ \eps (g_\theta(w)-\langle g\rangle (w))
\end{equation}
where we denote
\begin{equation}\label{def:gtheta}
g_\theta(w) = \int_0^\theta \Big( f_\tau(w) - \langle f \rangle(w) \Big) d \tau,
\end{equation}
while the associated vector field $F^{[1]}$ is defined by $$F^{[1]}(v)=\langle f \rangle(v).$$
We can then apply any standard non-stiff integrator to the micro-macro system \eqref{eq:micromacro}, for instance the simplest Euler method.
To derive a uniformly accurate method of order $n$, we follow the same methodology and consider the change of variable 
$\Phi^{[n]}$ and the vector field $F^{[n]}$ defined by the iterations
$$F^{[n]}=\langle f \circ \Phi^{[n]} \rangle,$$
$$
\Phi^{[n+1]}_\theta = {\rm id} +  \eps \int_0^\theta \hskip-1ex \Big(f_\tau \circ \Phi^{[n]}_\tau - \partial_u \Phi^{[n]}_\tau F^{[n]} \Big) d\tau 
-\eps \left\langle \int_0^\theta \hskip-1ex \Big(f_\tau \circ \Phi^{[n]}_\tau - \partial_u \Phi^{[n]}_\tau F^{[n]} \Big) d\tau\right\rangle$$
and then apply a standard non-stiff integrator of order $n$ to the micro-macro system \eqref{eq:micromacro}.

\subsection{Stroboscopic averaging pullback methods}
\label{subsection:pb}
In this section, we construct two changes of variable for the pullback formulation \eqref{eq:pullback}
with order one and two respectively, for the construction of efficient geometric schemes that are uniformly accurate w.r.t.\,$\eps$.
The proposed changes of variables are perturbations of the maps $\Phi^{[1]}_\theta$ and $\Phi^{[2]}_\theta$
defined in \eqref{eq:phiiter} and chosen in order to preserve quadratic first integrals of the original system \eqref{eq:HOP}. Their construction relies on the implicit midpoint rule, known to preserve quadratic first integrals.

\subsubsection{A first order change of variable} \label{sec:firstorder}
We observe that the change of variable 
$$ \Phi^{[1]}_\theta(v) = v+ \eps g_\theta(v)$$
where $g_\theta$ is given by (\ref{def:gtheta}), is $1$-periodic with respect to $\theta$, and it may sound attractive
because it is completely explicit.
However, considering such an explicit change of variable is not desirable because it destroys the preservation
of quadratic first integrals, and the Hamiltonian structure if $f_\tau$ is assumed Hamiltonian, i.e. 
the transformed pullback system \eqref{eq:pullback} looses these geometric properties in general.
Alternatively, we shall consider the following change of variable based on the implicit midpoint rule,
\begin{equation}\label{eq:eulmid}
\tilde \Phi^{[1]}_\theta(v) = v+ \eps g_\theta\left(\frac{v+\tilde \Phi^{[1]}_\theta(v)}2\right).
\end{equation}
We note that $\tilde \Phi^{[1]}_\theta$ as well as differentials of $\tilde \Phi^{[1]}_\theta$ involved in \eqref{eq:pullback} can be computed by fixed point iterations based on the following identities.
The quantities involved in the right-hand side of \eqref{eq:pullback} can be obtained as
\begin{align}
\partial_\theta \tilde  \Phi^{[1]}_\theta(v) &= \eps \big( f_\theta - \langle f \rangle \big)\left(\frac{v+\tilde \Phi^{[1]}_\theta(v)}2\right) + \frac\eps2 \partial_u g_\theta \left(\frac{v+\tilde \Phi^{[1]}_\theta(v)}2\right) \partial_\theta \tilde  \Phi^{[1]}_\theta(v), \label{eq:deltaphi} \\
(  \partial_u \tilde \Phi^{[1]}_\theta(v))^{-1} K &= 
K - \frac\eps2 \partial_u g_\theta\left(\frac{v+\tilde \Phi^{[1]}_\theta(v)}2\right) (K+ (  \partial_u \tilde \Phi^{[1]}_\theta(v))^{-1} K)
\label{eq:dwphi}
\end{align}
This permits us to propose an algorithm for computing the transformed vector field involved in (\ref{eq:pullback}) and defined as
\begin{equation}\label{eq:defFeps}
\tilde F_{\eps,\theta}(v) = \Big( \partial_u \tilde \Phi^{[1]}_{\theta}(v) \Big)^{-1} \Big(f_{\theta} \circ \tilde \Phi^{[1]}_{\theta} (v) - \frac{1}{\eps} \partial_\theta \tilde \Phi^{[1]}_{\theta}(v) \Big).
\end{equation}
The resulting algorithm for the computation of $\dot v=\tilde F_{\eps,t/\eps}(v)$ at time $t$ and point $v$ as given by equation (\ref{eq:micromacro}) is detailed below. In practice, it assumes that the function $(\theta,v) \mapsto f_\theta(v)$ and its directional derivative $\partial_u f_\theta(v) P$ in the direction $P$ are given and that their Fourier coefficients as periodic functions of $\theta$ can be computed cheaply, for instance by using a Fast Fourier Transform (FFT). 
We compute by fixed point iterations the approximations $P\simeq \tilde \Phi^\eps_\theta(v)$,
$U\simeq \frac{v+\tilde \Phi^{[1]}_\theta(v)}2$, $Q \simeq \frac{1}{\eps}\partial_\theta \tilde \Phi^{[1]}_\theta(v)$, $S\simeq \tilde F_{\eps,\theta}(v)$.

\begin{algorithm}[H]
{\bf Algorithm 1} for computing the vector field $\tilde F_{\eps,\theta}(w)$ in \eqref{eq:eulmid},\eqref{eq:defFeps} for given input $v,\theta,\eps$.   \\
Set $P^{[0]}=v$, $S^{[0]}=Q^{[0]}=0$, and for
$i=0,1,2,\ldots$ compute
\begin{eqnarray*}
U^{[i]} &=& \frac{v + P^{[i]}}2\\
P^{[i+1]} &=& v + \eps g_\theta(U^{[i]})\\
Q^{[i+1]} &=&  f_\theta(U^{[i]}) - \langle f(U^{[i]}) \rangle + \frac\eps2 \partial_u g_\theta(U^{[i]}) Q^{[i]}\\
K^{[i+1]} &=& f_\theta(P^{[i+1]})-Q^{[i+1]} \\
S^{[i+1]} &=& K^{[i+1]}- \frac\eps2 \partial_u g_\theta(U^{[i]}) (S^{[i]} + K^{[i+1]})
\end{eqnarray*}
The output is $S^{[i]}$ that converges to $\tilde F_{\eps,\theta}(v)$ as $i$ grows to infinity.
\end{algorithm}
\subsubsection{A second order change of variable} 
Carrying one additional iteration of (\ref{eq:phiiter}) we obtain 
\begin{align*}
\Phi^{[2]}_\theta (v)& = v +  \eps \int_0^\theta \Big(f_\tau \circ \Phi^{[1]}_\tau (v)- (\partial_u \Phi^{[1]}_\tau(v)) \langle \partial_u \Phi^{[1]}(v) \rangle^{-1} \langle f \circ \Phi^{[1]} (v)\rangle  \Big) d\tau.
\end{align*}
Using the notation \eqref{def:gtheta}, we have $\Phi^{[1]}_\theta = {\rm id} +\eps g_\theta$ and $\Phi^{[2]}_\theta $ may be rewritten as
\begin{eqnarray} \label{eq:n0}
\Phi^{[2]}_\theta = {\rm id} +\eps g_\theta 
+ \eps ^2 \int_{0}^\theta \Big(\partial_u f_\tau g_\tau- \langle \partial_u f g\rangle -(\partial_u g_\tau -\langle \partial_u g\rangle) \langle f\rangle    
 \Big) d\tau  + \mathcal{O}(\varepsilon^3),
\end{eqnarray}
where we have used
\begin{eqnarray*} 
\partial_u \Phi_\tau^{[1]}={\rm id}+\eps\int_0^\tau \partial_u \left(f_s -\langle f\rangle\right) ds, \quad 
\langle\partial_u \Phi^{[1]}\rangle^{-1}={\rm id}-\eps\int_0^1 \int_0^\tau \partial_u \left(f_s -\langle f\rangle\right) ds d\tau + \mathcal{O}(\varepsilon^2),
\end{eqnarray*}
and 
\begin{eqnarray*} 
 \langle f\circ\Phi^{[1]} \rangle = \langle f \rangle+\eps\int_0^1 \int_0^\tau \partial_u f_\tau \left(f_s -\langle f\rangle\right)dsd\tau + \mathcal{O}(\varepsilon^2).
\end{eqnarray*}
Following the idea introduced in Section \ref{sec:firstorder} (in order to conserve exactly quadratic first integrals),
we search for a change of variables satisfying
\begin{eqnarray} \label{eq:mid}
\tilde \Phi^{[2]}_\theta(v) &=& v + \varepsilon h_\theta\left(\frac{\tilde \Phi^{[2]}_\theta(v)+v}2\right).
\end{eqnarray}
and satisfying $\tilde \Phi^{[2]}_\theta(v) = \Phi^{[2]}_\theta(v)  + \bigo(\eps^3)$.
Expansing  \eqref{eq:mid} and in Taylor series with respect to $\eps$,
$$
\tilde \Phi^{[2]}_\theta(v)  =  v + \varepsilon h_\theta(v) + \frac{\eps^2}2 \partial_u h_\theta (v)h_\theta(v) + \bigo(\eps^3),
$$
and comparing with \eqref{eq:n0}, we deduce the vector field $h_\theta$ given by 
\begin{eqnarray} \label{eq:defFtheta} 
h_\theta &=& g_\theta
+ \eps \int_{0}^\theta \Big(\partial_u f_\tau g_\tau- \langle \partial_u f g\rangle -(\partial_u g_\tau -\langle \partial_u g \rangle) \langle f\rangle    
 \Big) d\tau - \frac \eps2 \partial_u g_{\theta} g_\theta.
\end{eqnarray}
\begin{remark}
Since $\tilde \Phi^{[2]}_\theta=\Phi_\theta^{[2]} + \mathcal{O}(\varepsilon^3)$, the result of Theorem \ref{thm:regularity} with $n=2$ remains true.
\end{remark}
\begin{theorem}
Consider the pullback formulation \eqref{eq:pullback} with change of variable 
$\tilde \Phi^{[2]}_\theta(v)$ defined in  \eqref{eq:mid}.
If the original system \eqref{eq:HOP} has a quadratic first integral  $I(u(t))=I(u(0))$, then the pullback formulation \eqref{eq:pullback}
has the same quadratic first integral $I(v(t))=I(v(0))$.
If $f_\theta(u) = J^{-1}\nabla_u H_\theta(u)$ in \eqref{eq:HOP} is a Hamiltonian vector field, then the pullback formulation \eqref{eq:pullback} is again Hamiltonian.
\end{theorem}
\begin{proof}
The advantage of formulation \eqref{eq:mid}-\eqref{eq:defFtheta} becomes apparent if one rewrites $h_\theta$ as
\begin{eqnarray} \label{eq:Ftheta}
h_\theta &=& g_\theta
+ \frac\eps2 \int_{0}^\theta \left([f_\tau+\langle f\rangle,g_\tau] - \big\langle[f+\langle f\rangle,g]\big\rangle\right) d\tau, 
\end{eqnarray}
where we expand
$$
-\partial_u g_\theta g_\theta = \int_0^\theta (\partial_u g_\tau \langle f\rangle + \langle \partial_u f\rangle g_\tau -\partial_u g_\tau f_\tau-\partial_u f_\tau g_\tau) d\tau.
$$
Here, $[\cdot,\cdot]$ denotes the usual Lie-bracket of two vector fields, i.e.
$$
[k_1,k_2](v) = \partial_u k_1 (v) k_2(v) - \partial_u k_2(v) k_1(v). 
$$ 
As a matter of fact, the Lie-bracket operation is inherently geometric and $F_\theta$ 
preserves first integrals as soon as $f_\theta$ does, and it is Hamiltonian as soon as $f_\theta$ is. 
Since the midpoint rule is known to preserve quadratic first integrals and to be a symplectic method for Hamiltonian vector fields, then $\tilde \Phi^{[2]}_\theta$ preserves these properties. 
\end{proof}
We then have the following formulas for the computation of the right-hand side of equation (\ref{eq:pullback}):  
\begin{align}
\tilde \Phi^\eps_\theta(v) &= v + \varepsilon h_\theta\left(\frac{\tilde \Phi^{[2]}_\theta(v)+v}2\right), \label{eq:n1}\\
\partial_\theta \tilde \Phi^{[2]}_\theta(v) &= \varepsilon \partial_\theta h_\theta\left(\frac{\tilde \Phi^{[2]}_\theta(v)+v}2\right) + \frac{\varepsilon}{2}  \partial_u h_\theta \left(\frac{\tilde \Phi^{[2]}_\theta(v)+v}2\right) \partial_\theta \tilde \Phi^{[2]}_\theta(v), \label{eq:n2}\\
(\partial_u \tilde  \Phi^{[2]}_\theta(v))^{-1}K &= K - \varepsilon \partial_u h_\theta \left(\frac{\tilde \Phi^{[2]}_\theta(v)+v}2\right)\frac{K+(\partial_u \tilde  \Phi^{[2]}_\theta(v))^{-1}K}2. \label{eq:n5}
\end{align}
 The resulting algorithm for the computation of $\dot w$ at time $t$ and point $w$ as given by equation (\ref{eq:micromacro}) is detailed below\footnote{%
To facilitate the convergence of the fixed point iterations for large $\eps$, for all $\eps$ one can multiply $g_\theta$ in \eqref{eq:eulmid} or $h_\theta$ in \eqref{eq:n1} by a damping term such as $\exp({-\eps^2})=1+\bigo(\eps^2)$ which does not affect the uniform accuracy of order two.
}
:
\begin{algorithm}[H]
{\bf Algorithm 2} for computing the vector field $\tilde F_{\eps,\theta}(v)$ in \eqref{eq:mid}, \eqref{eq:defFeps} for given input $v,\theta,\eps$.   \\
Set $P^{[0]}=v$, $S^{[0]}=Q^{[0]}=0$, and for
$i=0,1,2,\ldots$ compute
\begin{eqnarray*}
U^{[i]} &=& \frac{v + P^{[i]}}2\\
P^{[i+1]} &=& v + \eps h_\theta(U^{[i]})\\
Q^{[i+1]} &=&  \partial_\theta h_\theta + \frac\eps2 \partial_u h_\theta (U^{[i]}) Q^{[i]}\\
K^{[i+1]} &=& h_\theta(P^{[i+1]})-Q^{[i+1]} \\
S^{[i+1]} &=& K^{[i+1]} - \frac\eps2 \partial_u h_\theta (U^{[i]}) (S^{[i]} + K^{[i+1]}).
\end{eqnarray*}
The output is $S^{[i]}$ that converges to $\tilde F_{\eps,\theta}(v)$ as $i$ grows to infinity.
\end{algorithm}

\begin{remark}
The above methodology is not restricted to the preservation of quadratic first integrals.
Consider the special case where the vector field in \eqref{eq:HOP} has the form 
$$
f_\theta(u) = S_\theta \nabla I(u)
$$
where $S_\theta=-S_\theta^T$ is a skew-symmetric matrix of size $d\times d$ depending periodically on the variable $\theta\in \T$.
In this case, $I:\R^d\rightarrow \R$ turns out to be a first integral and needs not to be quadratic.
Then, in place of the implicit midpoint rule in \eqref{eq:mid}, one can consider the averaged vector field method \cite{QuM08} defined as
\begin{eqnarray} \label{eq:avf}
\tilde \Phi^{[2]}_\theta(v) &=& v + \varepsilon \int_0^1 h_\theta\left(\mu\tilde \Phi^{[2]}_\theta(v)+(1-\mu)v\right)d\mu,
\end{eqnarray}
and the corresponding pullback formulation \eqref{eq:pullback} has the same first integral ($I(v(t))=I(v(0))$).
\end{remark}

\section{Numerical experiments}
\label{sect:numerical}
In this section, we present some numerical experiments that confirm our theoretical analysis and illustrates the efficiency of our strategy. We test our method on two different models, the H\'enon-Heiles equations and the Klein-Gordon equation. 
\subsection{The numerical schemes}
To present our numerical schemes, let us write the standard averaging micro-macro equations defined in Subsection \ref{subsection:mm} as well as the stroboscopic averaging pullback equation defined in Subsection \ref{subsection:pb} under the following generic form:
\begin{equation}\dot u=G^{[n]}_{\eps,t/\eps}(t,u),\label{eq:generic}\end{equation}
where the superscript $n$ refers to the order of the  change of variable $\Phi^{[n]}$ used in the method.
Recall that the time derivatives up to order $n$ of the vector field $G^{[n]}_{\eps,t/\eps}$ are uniformly bounded with respect to $\eps\in ]0,1]$.
In our numerical experiments, we shall use the following numerical schemes that
are of integral type. This enables to gain one order of regularity and to use the simpler change of variable $\Phi^{[n-1]}$ for schemes of order $n$, instead of the change of variable $\Phi^{[n]}$ normally required in Theorem \ref{thm:regularity}.\footnote{Indeed, it can be shown that the local error and hence the global error of order two of the integral schemes $S^{RK2}_{\Delta t}$ and 
$S^{midpoint}_{\Delta t}$ applied to any system $\dot y=f(t,y)$ involves only the norms of the second order derivative $\ddot y$ of the solution and the derivatives $\partial_t f,\partial_y f,\partial_{yy} f,\partial_{yt} f$ of the vector field, and in contrast to a standard Runge-Kutta method, it does not depend on $\dddot y,\partial_{tt} f$.}
\bigskip

\noindent{\em -- Order 2 Integral Runge-Kutta scheme (explicit):} $u_{k+1}=S^{RK2}_{\Delta t}u_k$, with
$$S^{RK2}_{\Delta t}u_k=u_k+\int_{t_k}^{t_{k+1}}G^{[1]}_{\eps,t/\eps}\left(t_{k+1/2},u_k+\int_{t_k}^{t_{k+1/2}}G^{[1]}_{\eps,s/\eps}(t_k,u_k)ds\right)dt,$$
where $t_k=k\Delta t$. This scheme will be applied to the standard averaging micro-macro equations.
\bigskip

\noindent{\em -- Order 2 Integral midpoint scheme (implicit):} $u_{k+1}=S^{midpoint}_{\Delta t}u_k$, with
$$S^{midpoint}_{\Delta t}u_k=u_k+\int_{t_k}^{t_{k+1}}G^{[1]}_{\eps,t/\eps}\left(t_{k+1/2},\frac{u_k+u_{k+1}}2\right)dt.$$
This implicit scheme can be simply implemented using fixed point iterations. It will be applied to the stroboscopic averaging pullback equation.
Recall that $G^{[n]}_{\eps,\theta}$ is periodic w.r.t.\,$\theta$, so in these two methods the integrals can be either analytically precomputed, or easily computed numerically by using FFT. 
Moreover, it can be shown that, thanks to the uniform boundedness of the first time derivative of $G^{[1]}_{\eps,t/\eps}$, these two schemes are uniformly accurate (UA) of order 2.
\bigskip

\noindent{\em -- Order 3  extrapolation scheme :}
If $S^{RK2}_{\Delta t}$ denotes the above integral RK2 scheme constructed with $G^{[2]}_{\eps,t/\eps}$ instead of $G^{[1]}_{\eps,t/\eps}$, we construct a UA method $S^{extrap 3}_{\Delta t}$ of order 3 by implementing the Richardson extrapolation:
$$S^{extrap 3}_{\Delta t}=\frac{1}{3}\left(4S^{RK2}_{\Delta t/2}\circ S^{RK2}_{\Delta t/2}-S^{RK2}_{\Delta t}\right).$$
This scheme will be applied to the standard averaging micro-macro equations.

\bigskip

\noindent{\em -- Order 3 composition scheme :}
If $S^{midpoint}_{\Delta t}$ denotes the above integral midpoint scheme constructed with $G^{[2]}_{\eps,t/\eps}$ instead of $G^{[1]}_{\eps,t/\eps}$, we construct a UA method $S^{compos 3}_{\Delta t}$ of order exactly 3 by composition, setting:
$$S^{compos 3}_{\Delta t}=S^{midpoint}_{\alpha_3 \Delta t}\circ S^{midpoint}_{\alpha_2 \Delta t} \circ S^{midpoint}_{\alpha_1 \Delta t},$$
where (for instance) $\alpha_2=-1.8$ and $\alpha_1>\alpha_3$ are uniquely defined such that the order three conditions hold, $\alpha_1+\alpha_2+\alpha_3=1,\alpha_1^3+\alpha_2^3+\alpha_3^3=0$. We emphasise that this non-symmetric composition scheme is considered for illustrating the theory only, because efficient high-order symmetric composition methods could also be considered.
This scheme will be applied to the stroboscopic averaging pullback equation.
\bigskip

\noindent{\em -- Order 4 extrapolation scheme :}
If $S^{RK2}_{\Delta t}$ denotes the above integral RK2 scheme constructed with $G^{[3]}_{\eps,t/\eps}$ instead of $G^{[1]}_{\eps,t/\eps}$, we construct a UA method $S^{extrap 4}_{\Delta t}$ of order 4 by implementing the Bulirsch and the Stoer extrapolations:
$$S^{extrap 4}_{\Delta t}=\frac{1}{12}\left(27 S^{RK2}_{\Delta t/3}\circ S^{RK2}_{\Delta t/3}\circ S^{RK2}_{\Delta t/3}-16S^{RK2}_{\Delta t/2}\circ S^{RK2}_{\Delta t/2}+S^{RK2}_{\Delta t}\right).$$
This scheme will be applied to the standard averaging micro-macro equations.


\subsection{The H\'enon-Heiles model - micro-macro method}
\label{Henon-Heiles}
In this subsection and the next subsection, we consider the H\'enon-Heiles model \cite{henon-heiles,gni}, parametrized with $\eps$. This is a Hamiltonian system with the following Hamiltonian energy:
$$H(q_1,q_2,p_1,p_2)=\frac{p_1^2}{2\eps}+\frac{p_2^2}{2}+\frac{q_1^2}{2\eps}+\frac{q_2^2}{2}+q_1^2q_2-\frac{1}{2}q_2^3.$$
When $\eps$ is small, the variables $q_1$, $p_1$ are highly oscillatory. Let us write the associated filtered system, satisfied by the variable $w(t)\in \RR^4$ defined by
$$\left\{\begin{array}{l}
\ds w_1(t)=\cos\left(\frac{t}{\eps}\right)q_1(t)-\sin\left(\frac{t}{\eps}\right)p_1(t),\\
\ds w_2(t)=q_2(t),\\
\ds w_3(t)=\sin\left(\frac{t}{\eps}\right)q_1(t)+\cos\left(\frac{t}{\eps}\right)p_1(t),\\
\ds w_4(t)=p_2(t).
\end{array}\right.
$$
This variable $w(t)$ satisfies the system
\begin{equation}
\frac{dw}{dt}(t)=f\left(\frac{t}{\eps},w(t)\right),
\label{hop2}
\end{equation}
with
$$\left\{\begin{array}{l}
\ds f_1(\theta,w)=2\sin \theta\left(w_1\cos\theta+w_3\sin\theta\right)w_2\\[1mm]
\ds f_2(\theta,w)=w_4\\[1mm]
\ds f_3(\theta,w)=-2\cos \theta\left(w_1\cos\theta+w_3\sin\theta\right)w_2\\[1mm]
\ds f_4(\theta,w)=-2\left(w_1\cos\theta+w_3\sin\theta\right)^2+w_2^2-w_2.
\end{array}\right.
$$
For the standard averaging micro-macro method, all the integrals in $\theta$ in our numerical schemes can be pre-computed analytically. We have implemented these computations with the software Maple. 

Let us present our numerical results obtained by the standard averaging micro-macro method. The initial data is $(0.12,0.12,0.12,0.12)$. In Figure \ref{fig1a}, we have represented the error at $T_{final}=1$ as a function of $\Delta t$, for different values of $\eps$, for our UA scheme $S^{RK2}_{\Delta t}$ of order 2, for our UA scheme $S^{extrap 3}_{\Delta t}$ of order 3 and for our UA scheme $S^{extrap 4}_{\Delta t}$ of order 4. The reference solution is computed with the Matlab ode45 function applied to the filtered system \eqref{hop2}. The error is independent of $\eps$ as shown in Figure \ref{fig1a} where the curves for different values of $\eps$ are nearly identical. This confirms that the schemes are uniformly accurate respectively of order 2, 3 and 4, as predicted by Theorem \ref{thm:regularity}.

\begin{figure}[htb]
{
		\centering
\begin{subfigure}[t]{0.5\linewidth}
\includegraphics[width=1.05\linewidth]{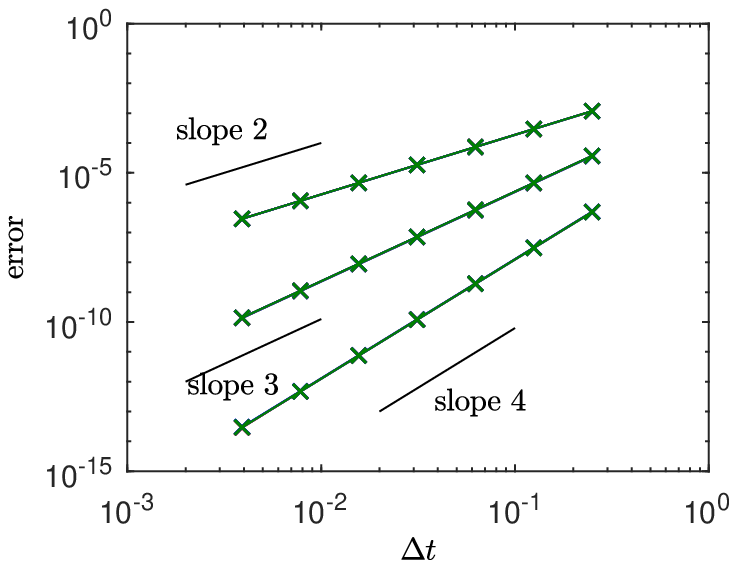}
\caption{Micro-macro method. \label{fig1a}}
\end{subfigure}
\begin{subfigure}[t]{0.5\linewidth}
\includegraphics[width=1.05\linewidth]{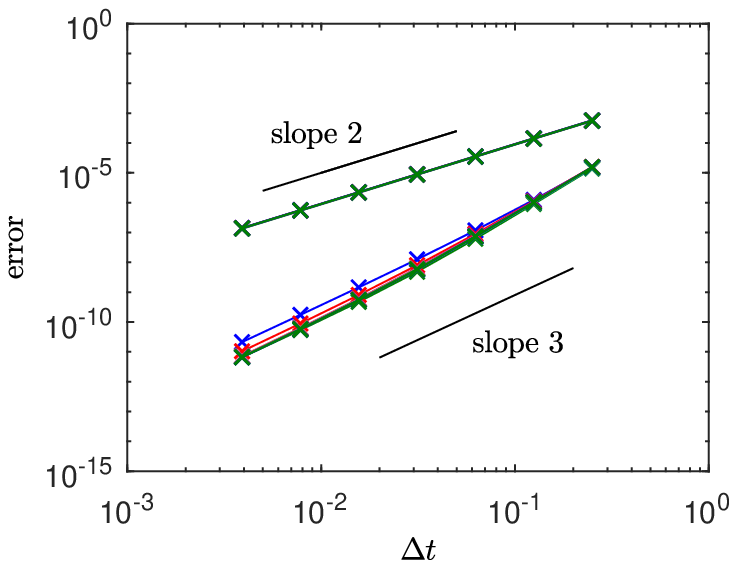}
\caption{Pullback method. \label{fig1b}}
\end{subfigure}
}
\caption{H\'enon-Heiles model. Error as a function of $\Delta t$ for $\eps=2^{-k}$, $k\in\{0,1,2,\cdots, 9\}$. 
Fig.\ts\ref{fig1a} (micro-macro method):
UA scheme $S^{RK2}_{\Delta t}$ of order 2 (top curves), UA scheme $S^{extrap 3}_{\Delta t}$ of order 3 (middle curves) and UA scheme $S^{extrap 4}_{\Delta t}$ of order 4 (bottom curves).
Fig.\ts\ref{fig1b} (pullback method):
UA scheme $S^{midpoint}_{\Delta t}$ of order 2 (top curves) and UA scheme $S^{compos 3}_{\Delta t}$ of order 3 (bottom curves).
}
\label{fig1}
\end{figure}

\subsection{The H\'enon-Heiles model - pullback method}

In this section, we present the results obtained with the pullback method with stroboscopic averaging for the H\'enon-Heiles model. In Figure \ref{fig1b}, we have represented the error at time $T_{final}=1$ as a function of $\Delta t$ for different values of $\eps$, for our UA scheme $S^{midpoint}_{\Delta t}$ of order 2 and for our UA scheme $S^{compos 3}_{\Delta t}$ of order 3. These figures confirm that our schemes are UA, as stated by Theorem \ref{thm:regularity}. Note that for these schemes, the variable $\theta$ of the vector field $F^{[n]}_{\eps, \theta}$ is discretized with 32 points and FFT is used to compute all the integrals\footnote{%
This discretization turns out to be exact for the H\'enon-Heiles problem, due to the low degree of the involved trigonometric polynomials.}. The initial data is $(0.12,0.12,0.12,0.12)$.

Let us now check that our schemes constructed on the stroboscopic averaging pullback method have the expected geometric properties. In Figure \ref{figHH2a}, we have represented the evolution of the error on the Hamiltonian $H$ for $\eps=1$ over long times $t \leq 30000$ for the scheme $S^{midpoint}_{\Delta t}$ with $\Delta t=0.1$ and $\Delta t=0.2$ and for the scheme $S^{compos 3}_{\Delta t}$ with $\Delta t=0.2$. In Figure \ref{figHH2b}, we have represented the results of the same calculations with $\eps=0.001$. The initial data is $(0,0,p_1,p_2)$ with $p_1=\sqrt{2\eps/12}\sin (\frac{\pi}{8})$, $p_2=\sqrt{2/12}\cos (\frac{\pi}{8})$. We observe no drift for the Hamiltonian.

\begin{figure}[htb]
{
		\centering
\begin{subfigure}[t]{0.5\linewidth}
\includegraphics[width=1.05\linewidth]{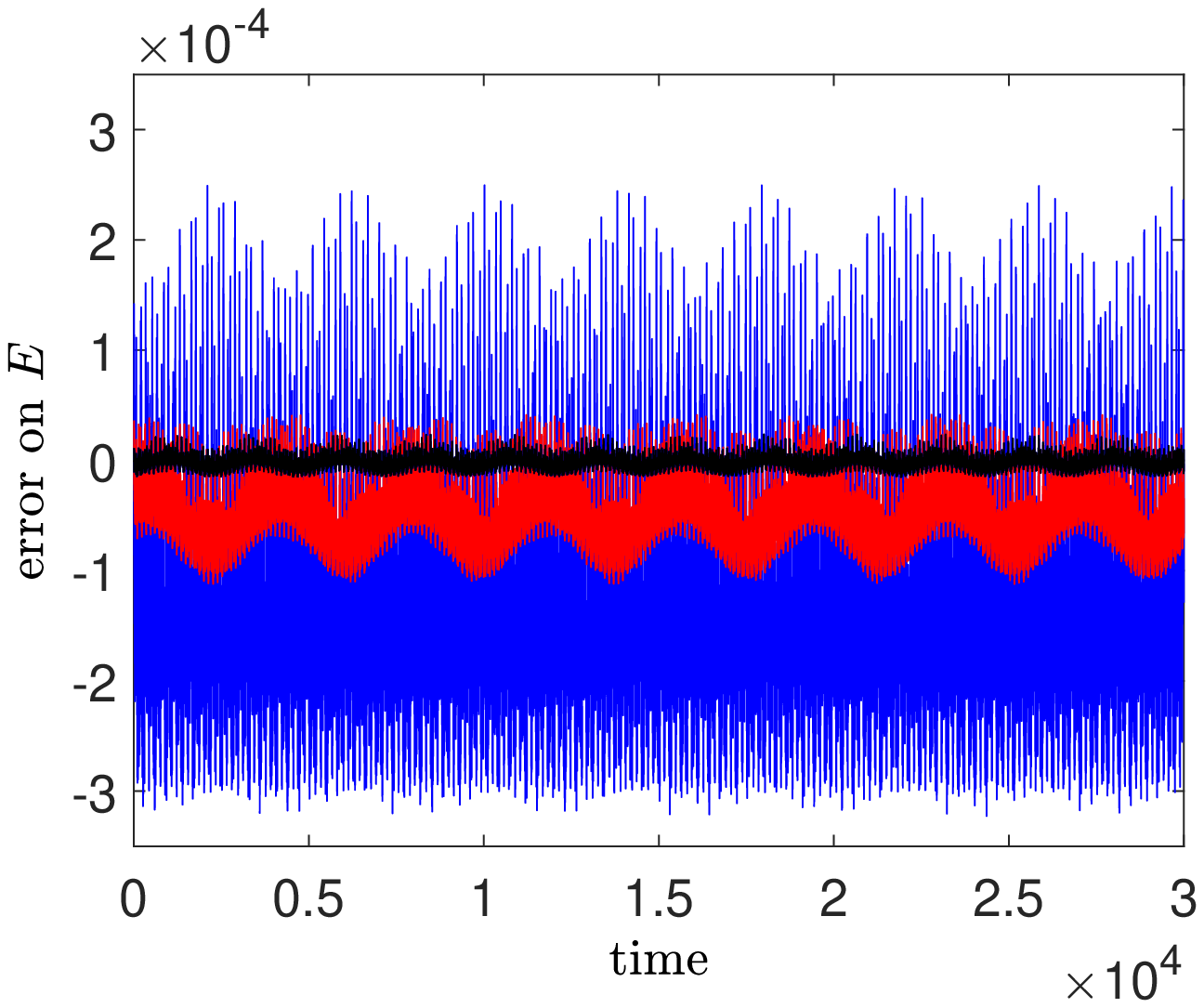}
\caption{$\eps=1$. } \label{figHH2a}
\end{subfigure}
\begin{subfigure}[t]{0.5\linewidth}
\includegraphics[width=1.05\linewidth]{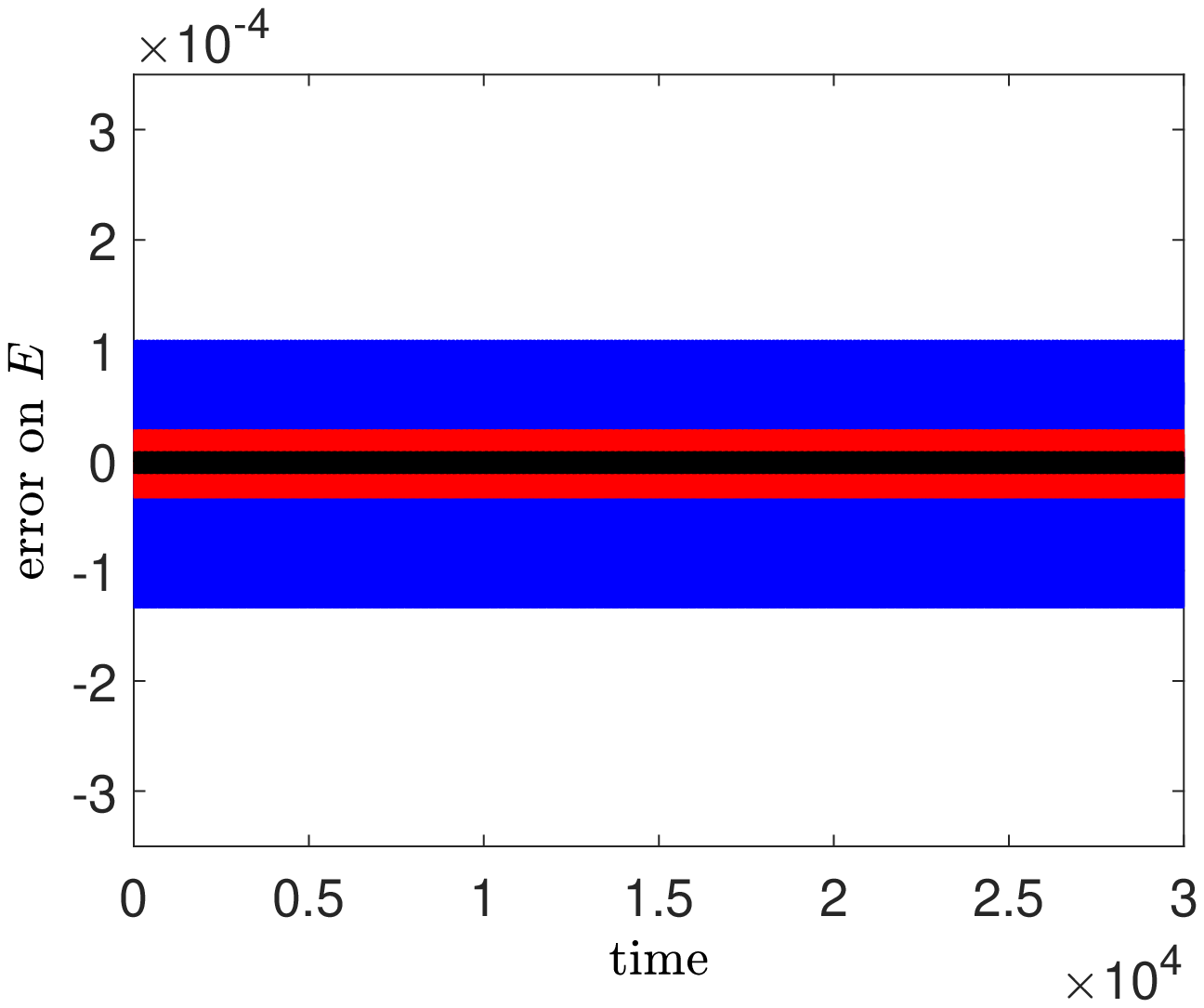}
\caption{$\eps=0.001$. \label{figHH2b}}
\end{subfigure}
}
\caption{H\'enon-Heiles model, pullback method. Long time evolution of the error on the Hamiltonian. UA scheme $S^{midpoint}_{\Delta t}$ of order 2 (blue: $\Delta t=0.2$; red: $\Delta t=0.1$). UA scheme $S^{compos 3}_{\Delta t}$ of order 3 (black: $\Delta t=0.2$).}\label{figHH2}
\end{figure}

Let us now observe the so-called Poincar\'e cuts \cite{henon-heiles,gni}, i.e. the points $(q_2,p_2)$ reached when the solution crosses the hyperplane $q_1=0$ with $p_1>0$. 

In Figure \ref{figpoincarea}, we have represented 5 orbits for $\eps=1$, $0\leq t\leq 30000$ and for the energy $H=1/12$. The initial data are $(0,0,p_1,p_2)$ with $p_1=\sqrt{2\eps H}\sin (\theta \frac{\pi}{2})$, $p_2=\sqrt{2H}\cos (\theta\frac{\pi}{2})$ with $\theta=0.025,\,0.25,\,0.75,\,0.95$ and $1.25$. Our UA scheme $S^{midpoint}_{\Delta t}$ is used with the timestep $\Delta t=0.1$ for $t\leq 30000$.

In Figure \ref{figpoincareb}, we have represented 3 orbits for $\eps=0.001$, $0\leq t\leq 30000$ and for the energy $H=1/6$. The initial data are $(0,0,p_1,p_2)$ with $p_1=\sqrt{2\eps H}\sin (\theta \frac{\pi}{2})$, $p_2=\sqrt{2H}\cos (\theta\frac{\pi}{2})$ with $\theta=0.1,\,0.5$ and $0.9$. Our UA scheme $S^{midpoint}_{\Delta t}$ is used with the timestep $\Delta t=0.1$ for $t\leq 30000$.

On these two figures, we observe that the Poincar\'e curves are well reproduced, which indicate that the "formal second invariant" is well preserved by our scheme.

\begin{figure}[htb]
{
		\centering
\begin{subfigure}[t]{0.5\linewidth}
\includegraphics[width=1.05\linewidth]{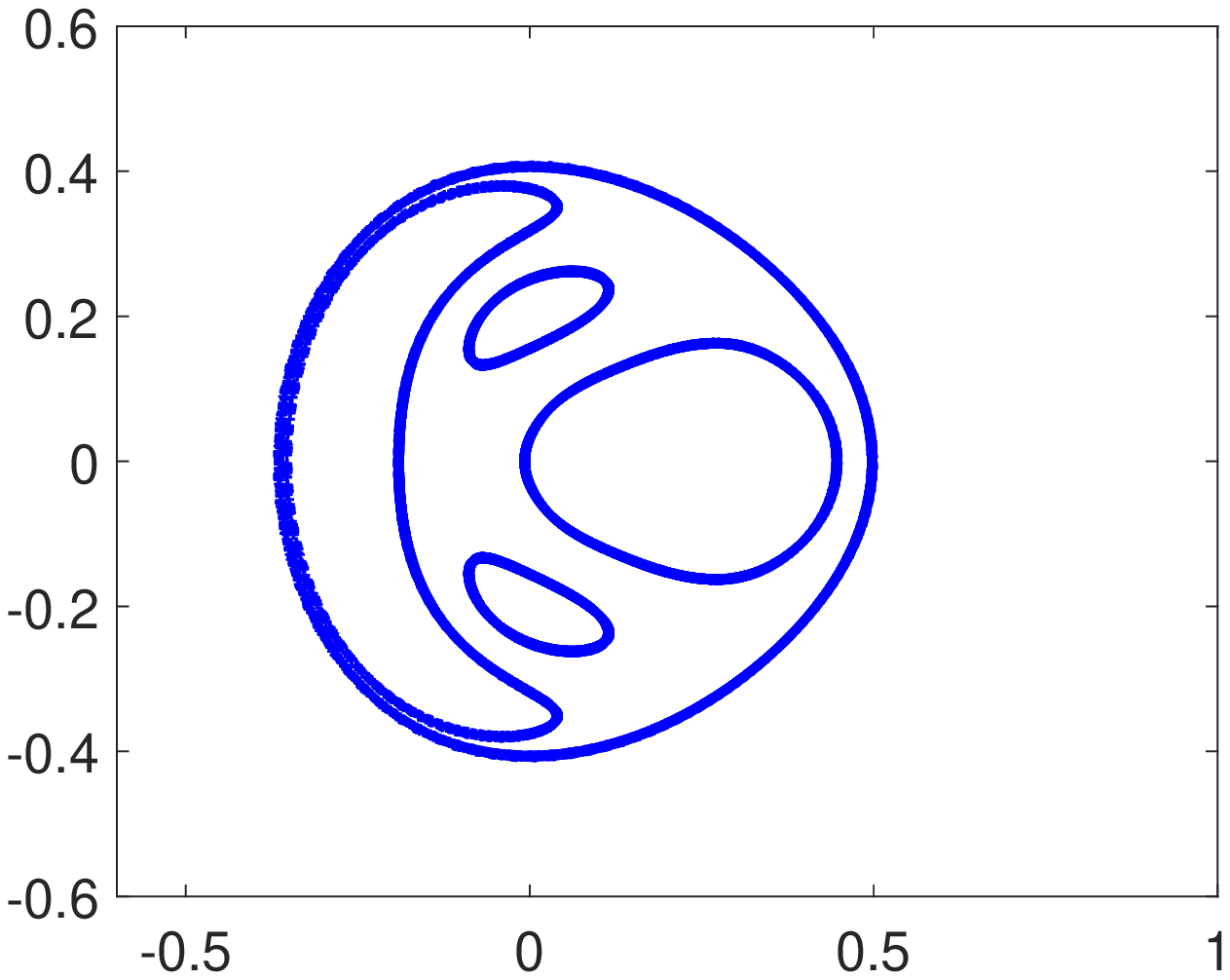}
\caption{$\eps=1$, $H=1/12$. \label{figpoincarea}}
\end{subfigure}
\begin{subfigure}[t]{0.5\linewidth}
\includegraphics[width=1.05\linewidth]{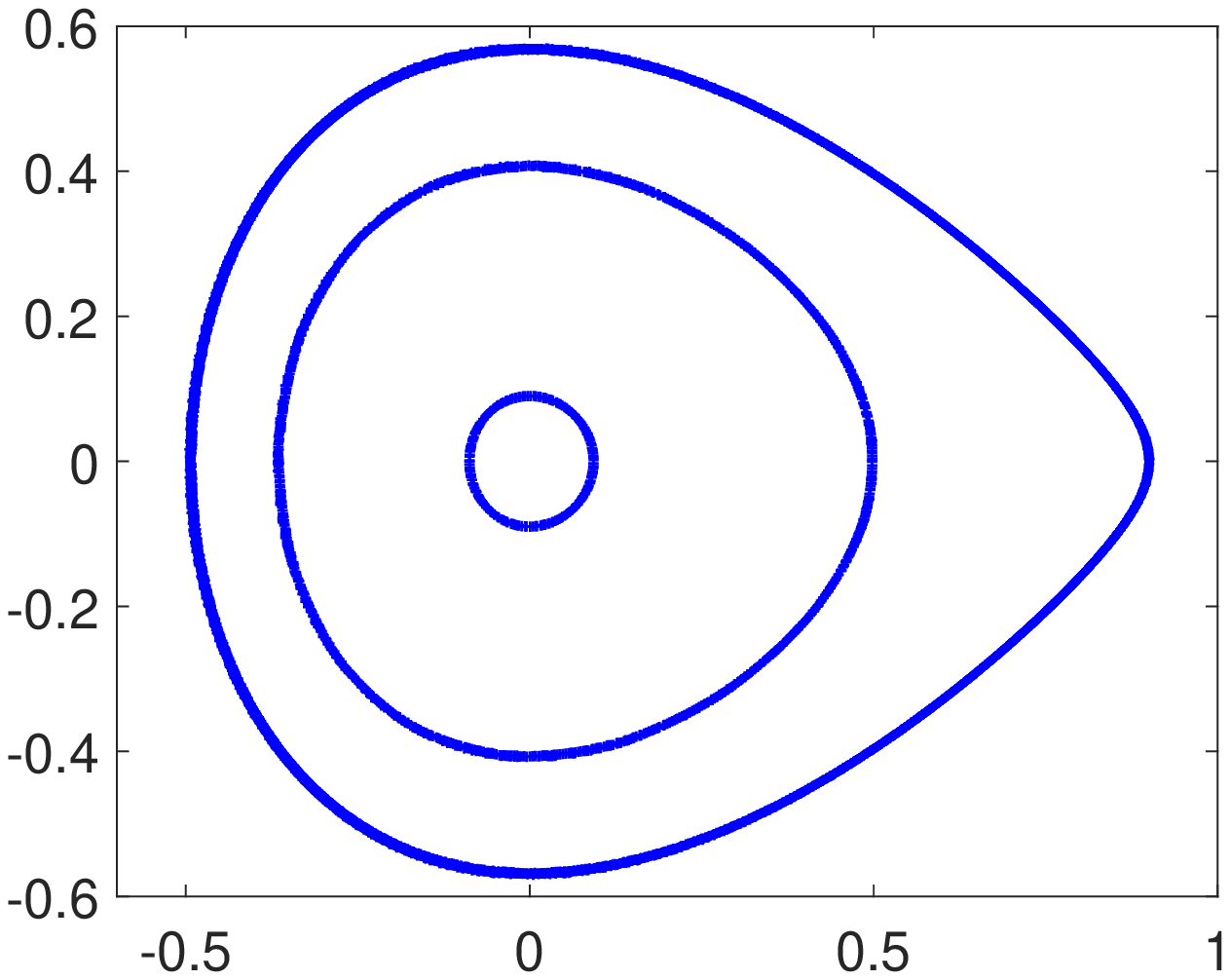}
\caption{$\eps=0.001$, $H=1/6$. \label{figpoincareb}}
\end{subfigure}
}
\caption{Poincar\'e cuts for the H\'enon-Heiles model. Pullback method with $S^{midpoint}_{\Delta t}$.}\label{figpoincare}
\end{figure}

\subsection{The nonlinear Klein-Gordon equation in the nonrelativistic limit regime - micro-macro method}

We consider the nonlinear Klein-Gordon (NKG) equation in the nonrelativistic limit regime, also studied numerically in \cite{bao-dong,faou-schratz2,cclm,faou-schratz}: 
\begin{equation}
\label{kg}
\eps \pa_{tt}u-\Delta u+\frac{1}{\eps}u+f(u)=0,\qquad x\in \RR^d,\quad t>0,
\end{equation}
with initial conditions given as
\begin{equation}
\label{kginit}
u(0,x)=\phi(x),\qquad \pa_tu(0,x)=\frac{1}{\eps}\gamma(x),\qquad x\in \RR^d.
\end{equation}
The unknown is the function $u(t,x):\,\RR^{1+d}\to \CC$ and the parameter $\eps>0$ is inversely proportional to the square of the speed of light. In our simulations, we take $d=1$ and the nonlinearity is $f(u)=4|u|^2u$. The limit $\eps\to 0$ in equation \eqref{kg}, \eqref{kginit}, referred to as the nonrelativistic limit, has been studied in \cite{kg1,kg2,kg3}. 

It is convenient to rewrite \eqref{kg} under the equivalent form of a first order system in time (see e.g. \cite{kg3}). Setting
\begin{equation}
\label{change}
v_+=u-i\eps(1-\eps\Delta)^{-1/2}\pa_t u,\qquad v_-=\overline{u}-i\eps(1-\eps\Delta)^{-1/2}\pa_t \overline{u},
\end{equation}
and denoting
$$f_\pm(v)=f\left(\frac{1}{2}(v_\pm+\overline{v_\mp})\right),$$
we obtain that \eqref{kg}, \eqref{kginit} is equivalent to the following system on the unknown $v=(v_+,v_-)$:
\begin{equation}
i\pa_t v_\pm=-\frac{1}{\eps}(1-\eps\Delta)^{1/2}v_\pm-(1-\eps\Delta)^{-1/2}f_\pm(v),
\label{kg2}
\end{equation}
\begin{equation}
\label{kg2init}
v(0,\cdot)=(v_+(0,\cdot),v_-(0,\cdot))=\left(\phi-i(1-\eps\Delta)^{-1/2}\gamma\,,\,\overline{\phi}-i(1-\eps\Delta)^{-1/2}\overline{\gamma}\right).
\end{equation}
Let us introduce the filtered unknowns
$$\widetilde v_\pm=e^{-i\frac{t}{\eps}\sqrt{1-\eps\Delta}}v_\pm.$$
These quantities satisfy the following equation:
\begin{equation}
\label{kg3}
\pa_t\widetilde v_\pm=i(1-\eps\Delta)^{-1/2}e^{-i\frac{t}{\eps}\sqrt{1-\eps\Delta}}f_\pm\left(e^{i\frac{t}{\eps}\sqrt{1-\eps\Delta}}\widetilde v\right).
\end{equation}
Note that \eqref{kg3} is under the form
$$ \pa_t\widetilde v_\pm(t)=f_\pm\left(\frac t\eps,t,\widetilde v(t)\right),$$ if we set
\begin{equation}
\label{cF1}
f_\pm(\theta,t,u)=i(1-\eps\Delta)^{-1/2}e^{-i\theta}e^{-itA_\eps}f_\pm\left(e^{i\theta}e^{itA_\eps}u\right)
\end{equation}
with
$$A_\eps=\frac{1}{\eps}\left(\sqrt{1-\eps\Delta}-1\right).$$
The system has two natural invariants:

\begin{itemize}
\item[--] the conserved charge (quadratic invariant)
\begin{align*}
Q=\IM \int_{\RR^d}\eps \dot u \,\overline{u}\,dx&=\frac{1}{4} \int_{\RR^d}(1-\eps\Delta)^{1/2}v_+\,\overline {v_+}\,dx-\frac{1}{4} \int_{\RR^d}(1-\eps\Delta)^{1/2}v_-\,\overline {v_-}\,dx\\
&=\frac{1}{4} \int_{\RR^d}(1-\eps\Delta)^{1/2}\widetilde v_+\,\overline {\widetilde v_+}\,dx-\frac{1}{4} \int_{\RR^d}(1-\eps\Delta)^{1/2}\widetilde v_-\,\overline {\widetilde v_-}\,dx
\end{align*}
\item[--] the energy
$$E=\eps \int_{\RR^d}|\dot u|^2\,dx+ \int_{\RR^d}|\nabla u|^2\,dx+\frac{1}{\eps} \int_{\RR^d}|u|^2\,dx+2 \int_{\RR^d}|u|^4\,dx.$$
\end{itemize}
\bs
In the simulations of this subsection, we take the following initial data:
\begin{equation}\label{initdat}
\phi(x)=\frac{1}{2-\cos x},\qquad \gamma(x)=\frac{1}{2-\sin x}.
\end{equation}
For Figure \ref{fig12}, the final time of the simulations is $T_{final}=0.25$. The computational domain in $x$ is $[0,2\pi]$. For the numerical evaluations of the function $f$ in our scheme, a spectral method is used in the $x$ variable and the fast Fourier transform is used in the practical implementation. For these tests, the reference solution is computed by our third order UA method $S^{extrap 3}_{\Delta t}$ applied to the standard averaging micro-macro method with 256 grid points in $x$, 128 grid points in $\theta$ and $\Delta t=T_{final}/4096$.
We use the $H^1$ relative error of a given numerical scheme which we define as
\begin{equation}
{\mathcal E}_{1}= \frac{\|u^{ref}(T_{final},\cdot)-u^{num}(T_{final},\cdot)\|_{H^1}}{\|u^{ref}(T_{final},\cdot)\|_{H^1}},
\end{equation}
where $u^{num}(t_{final},\cdot)$ is the approximated solution  obtained by the considered numerical scheme, at the final time $T_{final}$ of the simulation.

Let us present the results for the standard averaging micro-macro method. In this case, the analytic computation of the integrals in $\theta$ involves a too heavy system of equations and it is more convenient to compute these integrals numerically at each iteration of the scheme, by using fast Fourier transform techniques. In Figure \ref{fig12}, we have represented the error at $T_{final}=0.25$ as a function of $\Delta t$, for different values of $\eps$, for our UA scheme $S^{RK2}_{\Delta t}$ of order 2,  for our UA scheme $S^{extrap 3}_{\Delta t}$ of order 3 and for our UA scheme $S^{extrap 4}_{\Delta t}$ of order 4. These curves confirm that the schemes are UA respectively of order 2,  3 and 4. The numerical parameters are the following: we have taken 128 discretization points in the $x$ variable and $64$ discretization points in $\theta$ for the quadrature methods.

\begin{figure}[htb]
{
		\centering
\begin{subfigure}[t]{0.5\linewidth}
\includegraphics[width=1.05\linewidth]{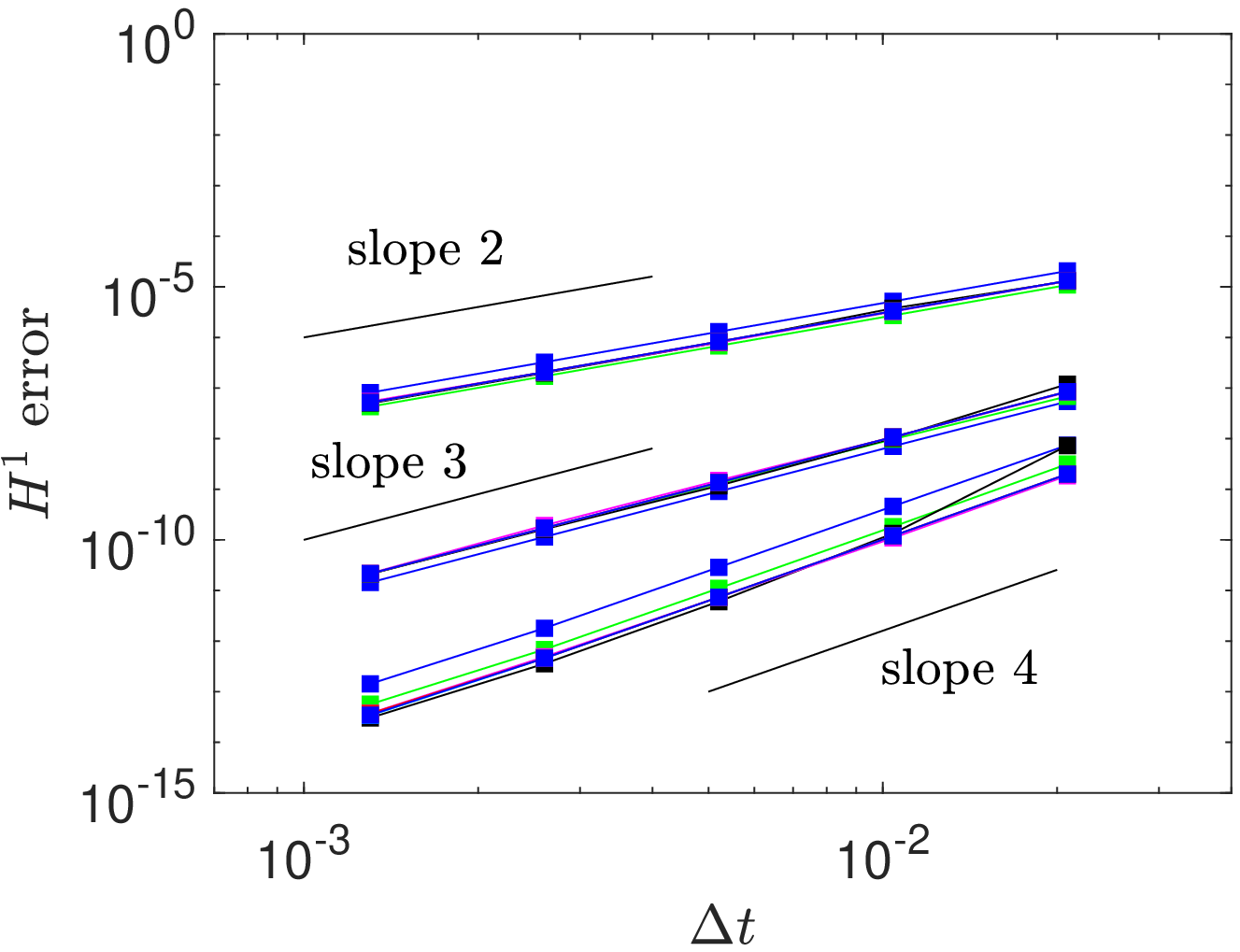}
\caption{Micro-macro method. \label{fig12}}
\end{subfigure}
\begin{subfigure}[t]{0.5\linewidth}
\includegraphics[width=1.05\linewidth]{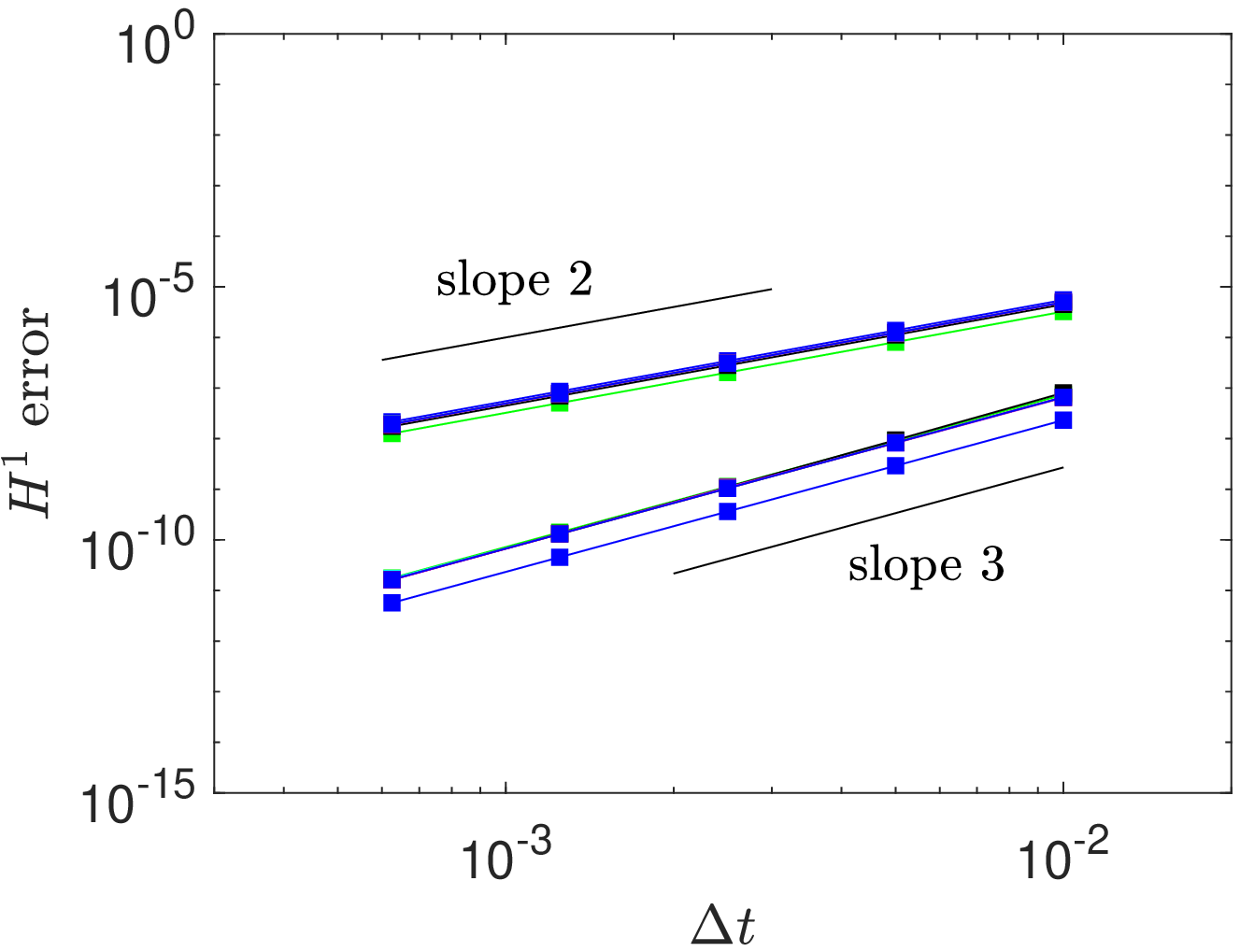}
\caption{Pullback method. \label{fig14}}
\end{subfigure}
}
\caption{NKG model. Error as a function of $\Delta t$ for  $\eps\in \{1, 10^{-1},10^{-2}, 10^{-3},10^{-4},10^{-5},10^{-6}\}$. 
Fig.\ts\ref{fig12} (micro-macro method):
UA scheme $S^{RK2}_{\Delta t}$ of order 2 (top curves), UA scheme $S^{extrap 3}_{\Delta t}$ of order 3 (middle curves) and UA scheme $S^{extrap 4}_{\Delta t}$ of order 4 (bottom curves).
Fig.\ts\ref{fig14} (pullback method):
scheme $S^{midpoint}_{\Delta t}$ of order 2 (top curves) and scheme $S^{compos 3}_{\Delta t}$ of order 3 (bottom curves).}\label{fig12zz}
\end{figure}


\subsection{The nonlinear Klein-Gordon equation in the nonrelativistic limit regime - pullback method}
\sloppy

We now present the results of the pullback method with the schemes $S^{midpoint}_{\Delta t}$ and $S^{compos 3}_{\Delta t}$. In Figure \ref{fig14}, we have plotted the error as a function of $\Delta t$ for different values of $\eps$. For this figure, we have discretized the $\theta$ variable with $64$ points and the $x$ variable with $128$ points, the initial data is \eqref{initdat} and the final time of the simulation is $T_{final}=0.25$. These numerical experiments confirm that our schemes $S^{midpoint}_{\Delta t}$ and $S^{compos 3}_{\Delta t}$ are respectively UA of order 2 and 3.


In Figure \ref{fig16}, we have represented the long time ($t\leq 1000$) conservation of the quadratic invariant $Q$ for our two schemes $S^{midpoint}_{\Delta t}$ and $S^{compos 3}_{\Delta t}$ and in Figure \ref{fig17} we have plotted the evolution of the relative error on the energy. The maximal absolute value of the error on $E$ is $10^{-8}$ for the second order scheme and is $2\times 10^{-10}$ for the third order scheme. The initial data, with a non trivial value of $Q$, is 
$$\phi(x)=(1+i)(\cos x+\sin x),\qquad \gamma(x)=(1-i/2)\cos x+(1/2+i)\sin x.$$
The blue curves correspond to $S^{midpoint}_{\Delta t}$ with $\Delta t=0.01$ and the red curves correspond to $S^{compos 3}_{\Delta t}$ with $\Delta t=0.01$. In all the simulations, we have $\eps=0.0001$ and we have taken  64 points for the discretization of the $x$ variable as well as for the $\theta$ variable. The quadratic invariant $Q$ is an exact invariant (up to round-off errors) of the two schemes, and we observe no energy drift over long times.

\begin{figure}[htb]
{
		\centering
\begin{subfigure}[t]{0.5\linewidth}
\includegraphics[width=1.05\linewidth]{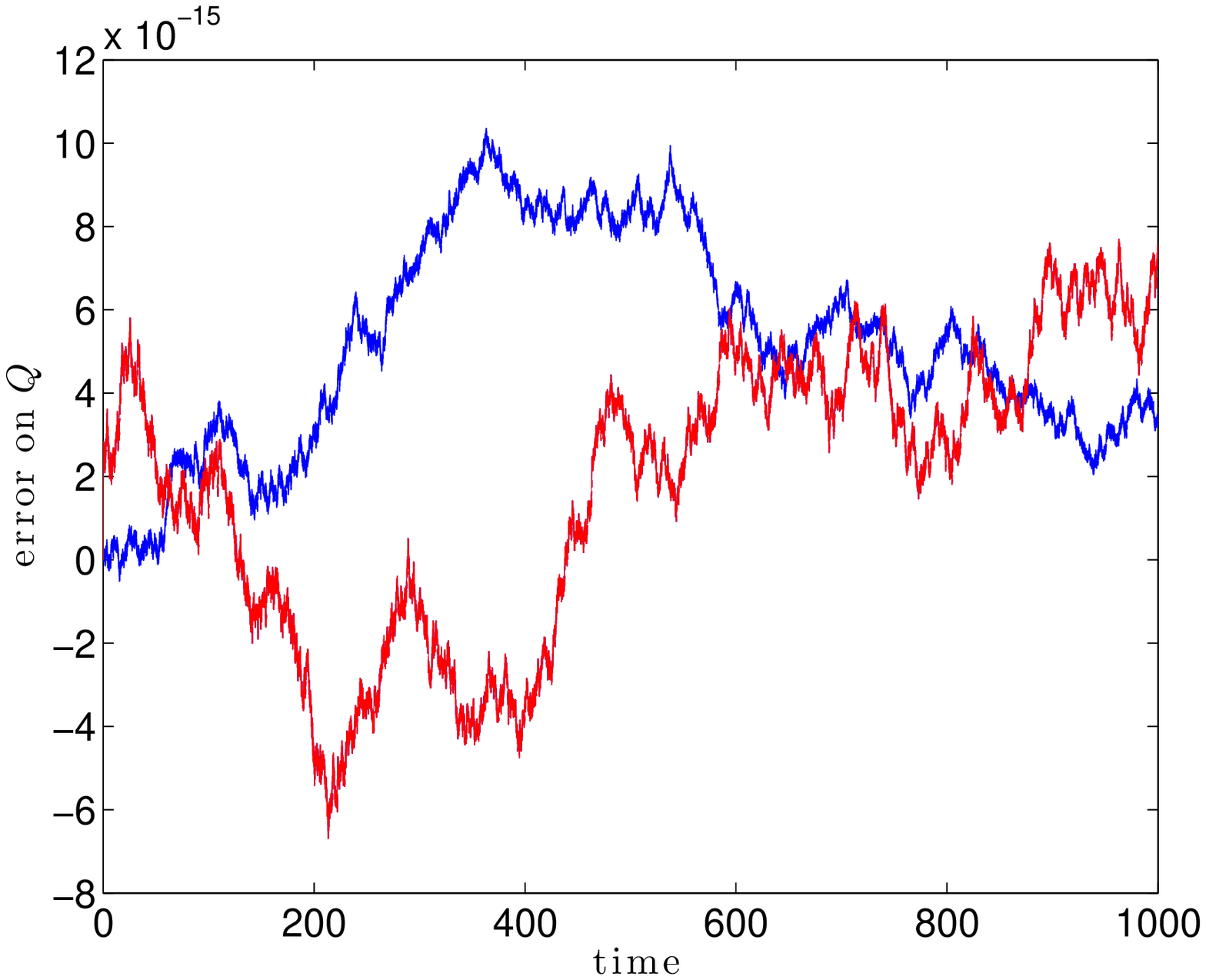}
\caption{quadratic invariant error. \label{fig16}}
\end{subfigure}
\begin{subfigure}[t]{0.5\linewidth}
\includegraphics[width=1.05\linewidth]{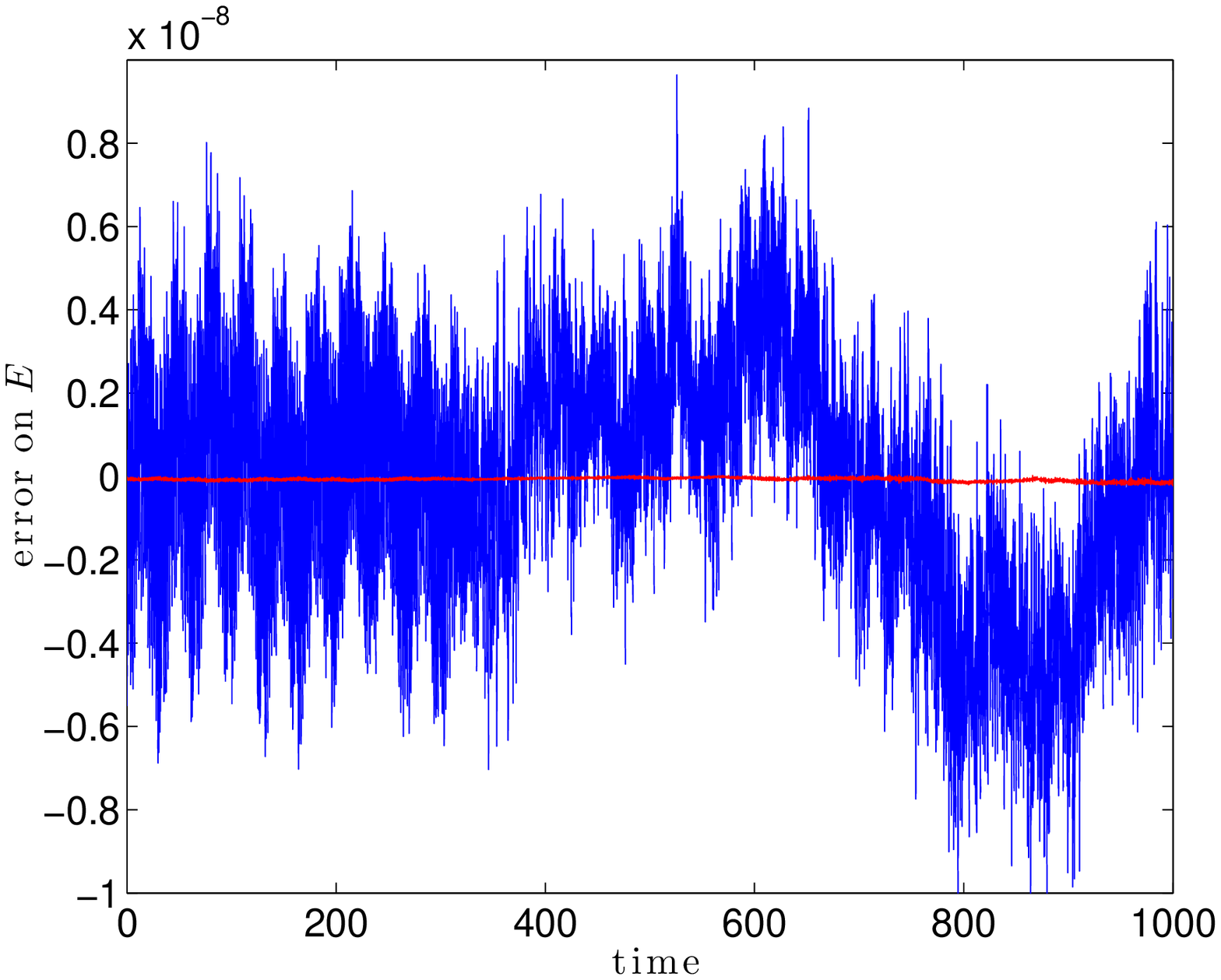}
\caption{Energy error. \label{fig17}}
\end{subfigure}
}
\caption{NKG model with $\eps=10^{-4}$, pullback method. Long-time evolution of the relative error on the quadratic invariant and on the energy. Blue: UA scheme $S^{midpoint}_{\Delta t}$ with $\Delta t=0.01$. Red: UA scheme $S^{compos 3}_{\Delta t}$ with $\Delta t=0.01$.}\label{fig17zz}
\end{figure}

\section{Conclusion}
We have shown that by transforming the original problem (\ref{eq:HOP}) through a change of variables inspired from the theory of averaging, it becomes possible to apply standard nonstiff numerical schemes (i.e. integrators which would properly solve equation  (\ref{eq:HOP}) with $\eps=1$) for all values of $\eps \in ]0,1]$ with a computational cost and accuracy both independent of $\eps\in]0,1]$. This has been done in four different ways whose merits have been discussed in details. Roughly speaking, stroboscopic averaging used in conjunction with the pullback method leads to geometric uniformly accurate integrators while standard averaging associated with a micro-macro decomposition leads to comparatively cheap uniformly accurate methods (which are not geometric though). As far as efficiency is concerned, the most relevant question is to estimate the critical threshold value $\eps_c$ below which our techniques become superior to standard schemes. Although such an $\eps_c$ clearly always exists, it is delicate to estimate its value  as it strongly depends on the problem itself and on the choices of implementation made in the various methods: for instance, we have exploited here the fact that both the change of variable and the averaged vector field can be computed at first and second orders {\em exactly} for the Henon-Heiles problem. Whenever this is possible, most of the additional formulae required by our techniques can be pre-computed, making the formers indeed very competitive even for large values of $\eps$ (i.e. $\eps_c$ is relatively large). The situation is different when this option is precluded (e.g. for the Klein-Gordon equation) and in that case the value of $\eps_c$ is undoubtedly  smaller.  However, we wish to emphasize that whatever its value might be, our technique remains of interest for variants of the model equation (\ref{eq:HOP}) where $\eps$ may undergo large variations. 
\bigskip

\noindent \textbf{Acknowledgements.}\
This work was partially supported by the Swiss National Science Foundation, grants No: 200020\_144313/1 and 200021\_162404 and by the ANR project Moonrise ANR-14-CE23-0007-01.

\appendix
\renewcommand{\thesection}{A}
\section*{Appendix}
This section is devoted to the proof of Proposition \ref{prop:mta2} and Theorem \ref{th:delta}. We first recall the two following basic results from \cite{ccmm}:
\begin{lemma} [See Castella, Chartier, Murua and M\'ehats \cite{ccmm}]\label{lem:basic}
Let $0 < \delta < \rho \leq 2R$. Assume that the function $(\theta,u) \in \T \times {\cal K}_\rho \mapsto \varphi_\theta(u) \in  X_\C$ is analytic, and that $\varphi_\theta$ is a near-identity mapping, in the sense that
$$
\|\varphi - {\rm Id}\|_\rho \leq \frac{\delta}{2}.
$$
Then $\|\partial_u \varphi \|_{\rho-\delta} \leq \frac32$, the mapping $\partial_u \langle \varphi \rangle ^{-1}$ is well-defined and analytic on ${\cal K}_{\rho-\delta}$ with 
$\|\partial_u \langle \varphi \rangle^{-1} \|_{\rho-\delta} \leq 2$, and 
the mappings $(\theta,u) \in \T \times {\cal K}_{\rho-\delta} \mapsto \Lambda (\varphi)_{\theta}(u)$ 
and $(\theta,u) \in \T \times K_{\rho-\delta} \mapsto \Gamma^\eps (\varphi)_{\theta}(u)$
are well-defined and analytic. In addition, the following bounds hold for all $\eps\geq 0$,
$$
\|\Lambda (\varphi)\|_{\rho-\delta} \leq 4 \, M \quad \mbox{ and } \quad \|\Gamma^\eps(\varphi) - {\rm Id}\|_{\rho-\delta} \leq 4  \, M  \, \eps.
$$
\end{lemma}
Lemma \ref{lem:basic} shows that, starting from a function  $(\theta,u) \in \T \times {\cal K}_{2R} \mapsto \varphi_\theta(u) \in  X_\C$,  we can consider iterates $\left(\Gamma^\eps\right)^k (\varphi)_\theta$
at the cost of a gradual thinning of their domains of analyticity.
The following contraction property holds.
\begin{lemma} [See Castella, Chartier, Murua and M\'ehats \cite{ccmm}] \label{lem:contract}
Let $0 < \delta < \rho \leq 2R$ and consider two periodic, near-identity mappings $(\theta,u) \in \T \times {\cal K}_\rho \mapsto \varphi_\theta(u)$ and $(\theta,u) \in \T \times {\cal K}_\rho \mapsto \tilde \varphi_\theta(u)$, analytic on ${\cal K}_\rho$ and satisfying 
$$
\|\varphi-{\rm Id}\|_\rho \leq \frac{\delta}{2} \quad \mbox {and} \quad\|\tilde \varphi-{\rm Id}\|_\rho \leq \frac{\delta}{2}. 
$$
Then the following estimates hold for all $\eps\geq 0$, 
$$
\|\Lambda (\varphi) - \Lambda (\tilde \varphi)\|_{\rho-\delta} \leq \frac{16 \, M}{\delta}  \|\varphi - \tilde \varphi\|_{\rho}
\quad  \mbox{and} \quad 
 \|\Gamma^\eps (\varphi) - \Gamma^\eps (\tilde \varphi)\|_{\rho-\delta}
\leq  \frac{16 \, M \eps}{\delta}  \|\varphi - \tilde \varphi\|_\rho.
$$
\end{lemma}
\begin{proof}[Proof of Proposition \ref{prop:mta2}]
We assume that all $\Phi^{[k]}$ are defined by (\ref{eq:phiiter}) in the stroboscopic case (all functions  $G^{[k]}$ are taken null). The following properties were already proved in \cite{ccmm}: 
\begin{itemize}
\item[(i)] the maps $\Phi^{[k]}$ and  $F^{[k]} := \langle  \partial_u \Phi^{[k]} \rangle^{-1} \langle f \circ  \Phi^{[k]} \rangle$ are well-defined and analytic on $K_{R_k}$ for all $k=0,\ldots,n+1$;
\item[(ii)] the estimates 
\begin{align*}
 \left\|\Phi^{[k]}-{\rm id}\right\|_{R_{k}} \leq \frac{r_n}{2}, \quad  \left\|F^{[k]}\right\|_{R_{k+1}} \leq 2M \; \mbox{ and } \; \left\|\Phi^{[k+1]}-\Phi^{[k]}\right\|_R \leq C \left( 2 (n+1) \frac{\eps}{\eps_0} \right)^k,
\end{align*}
hold for some positive constant $C$ independent of $n$, $k$ and $\eps$.
\end{itemize}
From equation (\ref{eq:phiiter}) with $G^{[k+1]} \equiv 0$ and Assumption \ref{ass:f}, it is obvious that $\Phi^{[k]}$ is $(p+1)$-times continuously differentiable w.r.t.\,$\theta$. We have furthermore, for $0 \leq k \leq n$
\begin{align*}
\partial_\theta \Phi^{[k+1]}_\theta = \eps \left(f_\theta \circ \Phi^{[k]}_\theta - 
 \partial_u \Phi^{[k]}_\theta F^{[k]}\right)
\end{align*}
so that, by differentiating $\nu \leq p$ times w.r.t.\,$\theta$ with the help of the Leibniz and the Fa\`a di Bruno formulae, we obtain  
\begin{align*}
\frac{1}{\eps} \partial^{\nu+1}_\theta \Phi^{[k+1]}_\theta &= \partial_\theta^\nu f_\theta \circ \Phi^{[k]}_\theta +  \sum_{q=1}^\nu \left(\hskip-1ex\begin{array}{c} \nu \\ q \end{array} \hskip-1ex\right) \sum_{\bf i} C^{q}_{\bf i} (\partial^{\nu-q}_\theta \partial^{|{\bf i}|}_u f_\theta) \circ \Phi^{[k]}_\theta
\left((\partial_\theta \Phi^{[k]}_\theta)^{i_1}, \ldots,(\partial^q_\theta \Phi^{[k]}_\theta)^{i_q} \right) \\
&-\partial_\theta^\nu \partial_u \Phi^{[k]}_\theta F^{[k]}
\end{align*}
where the inner sum runs over multi-indices ${\bf i}=(i_1,\ldots,i_q) \in \N^q$ such that $i_1+2i_2+\ldots+qi_q=q$ and where $|{\bf i}|=i_1+\ldots+i_q$ and  
$$
C^{q}_{\bf i}=\frac{q!}{i_1! i_2! 2!^{i_2}\ldots i_q!q!^{i_q}}.
$$ 
Owing to (ii), for $u \in {\cal K}_{R_{k+1}}$, $\Phi^{[k]}_\theta(u)$ takes its values in ${\cal K}_{R_{k+1}+r_n/2}$, a set on which $\partial_\theta^j \partial_u^i f_\theta$ is bounded (due to Assumption \ref{ass:f} and Cauchy estimates for analytic functions) as follows:
$$
\left\| \partial_\theta^j \partial_u^i f \right\|_{R_{k+1}+r_n/2} \leq \frac{1}{(2R-(R_{k+1}+r_n/2))^i} M = \frac{M}{(k+\frac12)^i r_n^i}.
$$
Hence, if we denote  $\alpha_{k,\nu}=\|\partial^\nu_\theta \Phi^{[k]}\|_{R_k}/(\eps M)$, for $\nu \geq 1$, then we have 
\begin{align*}
\forall k \geq 1, \quad\forall \nu \geq 1, \quad \alpha_{k+1,\nu+1} \leq 1 + \frac{2 \eps M  \alpha_{k,\nu}}{r_n} + \sum_{q=1}^\nu \left(\hskip-1ex\begin{array}{c} \nu \\ q \end{array} \hskip-1ex\right)
\sum_{\bf i} C^{q}_{\bf i} \frac{(\eps M)^{|\bf{i}|} \alpha_k^{\bf{i}}}{r_n^{|\bf i|} (k+\frac12)^{|\bf i|} } 
\end{align*}
where we denote $\alpha_k^{\bf{i}} =\prod_{j=1}^q 
\alpha_{k,j}^{i_j}$ and use that (owing to (ii) and a Cauchy estimate)
$$
\left\|\partial_u \partial_\theta^\nu \Phi^{[k]} \, F^{[k]} \right\|_{R_{k+1}}
\leq \frac{1}{r_n} \left\|\partial_\theta^\nu \Phi^{[k]} \right\|_{R_{k}} \left\|F^{[k]} \right\|_{R_{k+1}} \leq \frac{2 M  \left\|\partial_\theta^\nu \Phi^{[k]} \right\|_{R_{k}}}{r_n}.
$$
Finally, the assumption that $(n+1) \eps \leq \eps_0 = \frac{R}{8M}$ leads to the inequality 
\begin{align} \label{eq:ineqalpha}
\alpha_{k+1,\nu+1} \leq  1 + \frac{\alpha_{k,\nu}}{4}+ \sum_{q=1}^\nu \left(\hskip-1ex\begin{array}{c} \nu \\ q \end{array} \hskip-1ex\right)
\sum_{\bf i} C^{q}_{\bf i} \frac{\alpha_k^{\bf{i}}}{(8 k+4)^{|\bf i|} } .
\end{align}
We now introduce the generating functions for $\xi \in \C$,
$$
g_k(\xi) = \sum_{\nu \geq 1} \alpha_{k,\nu} \frac{\xi^\nu}{\nu!}, \quad k=1,\ldots
$$
A direct estimation gives $\alpha_{1,1} \leq  2$ and $\alpha_{1,\nu} \leq 1$ for $\nu \geq 2$, so that $0 \leq g_1(\xi) \leq \xi+e^\xi-1$ for all $\xi \in \R_+$ and from (\ref{eq:ineqalpha}), we obtain (owing to the Leibniz and the Fa\`a di Bruno formulae)
\begin{align*} 
\alpha_{k+1,\nu+1} \leq  \frac{\alpha_{k,\nu}}{4}+ \sum_{q=0}^\nu \left(\hskip-1ex\begin{array}{c} \nu \\ q \end{array} \hskip-1ex\right) \left.\frac{d^q}{d\xi^q}
\exp\left(\frac{g_k(\xi)}{8 k+4}\right)\right|_{\xi=0} = \frac{\alpha_{k,\nu}}{4} + G^{(\nu)}_k(0)
\end{align*}
where 
$$
G_k(\xi) = \exp\left(\xi+\frac{g_k(\xi)}{8 k+4}\right)
$$
so that, using $\alpha_{k+1,1} \leq 4$, we have  for $\xi \in \R_+$ 
$$
g_{k+1}(\xi) \leq 3 \xi + \frac14 \int_0^\xi g_k(\s) ds + \int_0^\xi G_k(s) ds.
$$
Noticing that $\sup_{0 \leq \xi \leq 1} g_k(\xi) = g_k(1)$, the previous inequality leads to 
$$
g_{k+1}(1) \leq 3 + \frac14  g_k(1) + (e-1) \exp\left(\frac{g_k(1)}{8 k +4}\right)
$$
and an easy induction shows that the sequence $(g_k(1))_{k \geq 1}$ is upper-bounded by $8$. We conclude the proof by applying Cauchy estimates as follows: 
$$
\forall 1 \leq k \leq n+1, \, \forall 1 \leq \nu \leq p+1, \quad  \alpha_{k,\nu} = \left| \frac{d^\nu g_k}{d\xi^\nu} (0)  \right| \leq  \frac{\nu! \, \sup_{|\xi| \leq 1} |g_k(\xi)|}{1^\nu} \leq  \nu! \, g_k(1) \leq   8 \, \nu!
$$
\end{proof}
\begin{proof}[Proof of Lemma \ref{lem:basicdiff}]
Under the assumptions (\ref{eq:nearidentity}), it has been shown in \cite{ccmm} that: 
\begin{enumerate}
\item[(a)] $\|\partial_u \varphi \|_{\rho-\delta} \leq \frac32$, the mapping $ \langle \partial_u \varphi \rangle ^{-1}$ is analytic on ${\cal K}_{\rho-\delta}$
and obeys $\| \langle \partial_u \varphi \rangle^{-1} \|_{\rho-\delta} \leq 2$;
\item[(b)] the mappings $\Lambda (\varphi)_{\theta}(u)$ 
and $\Gamma^\eps (\varphi)_{\theta}(u)$
are analytic on ${\cal K}_{\rho-\delta}$ for all $\theta \in \T$, and 
$$
\forall 0 < \eps < \eps_0, \quad \|\Lambda (\varphi)\|_{\rho-\delta} \leq 4 \, M \quad \mbox{ and } \quad \|\Gamma^\eps(\varphi) - {\rm Id}\|_{\rho-\delta} \leq 4  \, M  \, \eps.
$$
\end{enumerate} 
Now, on the one hand, for a multi-index ${\bf i}=(i_1,\ldots,i_q) \in \N^q$, the algebraic identity 
$$
\prod_{j=1}^q X_j^{i_j} - \prod_{j=1}^q Y_j^{i_j} = 
\sum_{j=1}^q (X_j-Y_j) \prod_{l=1}^{j-1} Y_l^{i_l} \prod_{l=j+1}^{q} X_l^{i_l} \sum_{r=1}^{i_j}  X_j^{i_j-r} Y_j^{r-1}  
$$
leads, for any $(\Delta_1, \ldots, \Delta_q) \in X_{\C}^q$ and any $(\tilde \Delta_1, \ldots, \tilde \Delta_q) \in X_{\C}^q$, to 
\begin{align} \label{eq:polar}
&\left\| \partial_u^{|\bf{i}|} f \circ \tilde \varphi \; (\Delta_1^{i_1}, \ldots,\Delta_q^{i_q}) - \partial_u^{|\bf{i}|} f \circ \tilde \varphi \;  (\tilde \Delta_1^{i_1}, \ldots,\tilde \Delta_q^{i_q})  \right\|_{\rho-\delta}  \nonumber \\
&\leq \| \partial_u^{|\bf{i}|} f \|_{\rho-\delta/2} \, \sum_{j=1}^{q} \| \Delta_j -\tilde \Delta_j\| \, \prod_{l=1}^{j-1} \|\Delta_l\|^{i_l} \, \prod_{l=j+1}^{q} \|\tilde \Delta_l\|^{i_l} \, \sum_{r=1}^{i_j} \|\Delta_j\|^{i_j-r} \|\tilde \Delta_j\|^{r-1}.
\end{align}
On the other hand, the Fa\`a di Bruno's formula with $1 \leq \nu < p$ gives  
\begin{align*}
\partial_\theta^\nu \left( f_\theta \circ \varphi_\theta - f_\theta \circ \tilde \varphi_\theta \right) & = \partial_\theta^\nu f_\theta \circ \varphi_\theta -\partial_\theta^\nu f_\theta \circ \tilde \varphi_\theta    \\
 + & \sum_{q=1}^\nu \left(\hskip -1ex \begin{array}{c}  \nu \\ q\end{array} \hskip -1ex 
\right) \sum_{\bf i}
C_{\bf i}^{q}  \partial^{\nu-q}_\theta \partial^{|{\bf i}|}_u f_\theta \circ \varphi_\theta
\left((\partial_\theta \varphi_\theta)^{i_1}, \ldots,(\partial^q_\theta \varphi_\theta)^{i_q} \right) \\
 - & \sum_{q=1}^\nu \left(\hskip -1ex \begin{array}{c}  \nu \\ q\end{array} \hskip -1ex 
\right) \sum_{\bf i}
C_{\bf i}^{q}  \partial^{\nu-q}_\theta \partial^{|{\bf i}|}_u f_\theta \circ \tilde \varphi_\theta
\left((\partial_\theta \tilde \varphi_\theta)^{i_1}, \ldots,(\partial^q_\theta \tilde \varphi_\theta)^{i_q} \right) 
\end{align*}
with the notations of Theorem \ref{prop:mta2}. Hence, using the decomposition in (\ref{eq:polar}) we have 
\begin{align*}
&\|\partial_\theta^\nu \left( f_\theta \circ \varphi_\theta - f_\theta \circ \tilde \varphi_\theta \right) \|_{\rho-\delta} \leq  \|\partial_u \partial_\theta^\nu f_\theta \|_{\rho-\delta/2} \;  \| \varphi - \tilde \varphi\|_{\rho-\delta}\\
& +  \sum_{q=1}^\nu  \left(\hskip -1ex \begin{array}{c}  \nu \\ q\end{array} \hskip -1ex 
\right) \sum_{\bf i} 
C_{\bf i}^{q} \;  \beta^{\bf i} \; \|\partial^{\nu-q}_\theta \partial^{|{\bf i}|+1}_u f \|_{\rho-\delta/2}  \, \| \varphi - \tilde \varphi\|_{\rho-\delta} \\
& + \sum_{q=1}^\nu \left(\hskip -1ex \begin{array}{c}  \nu \\ q\end{array} \hskip -1ex 
\right) \sum_{\bf i}
C_{\bf i}^{q} \;  \beta^{\bf i} \; \|\partial^{\nu-q}_\theta \partial^{|{\bf i}|}_u f \|_{\rho-\delta/2} 
\sum_{j=1}^q \frac{i_j}{\beta_j} \| \partial_\theta^j (\varphi - \tilde \varphi)\|_{\rho-\delta} 
\end{align*}
Upon using Cauchy estimates, we obtain 
\begin{align*}
\|\partial_\theta^\nu \left( f \circ \varphi - f \circ \tilde \varphi \right) \|_{\rho-\delta} \leq & \frac{2M}{\delta} \;  \| \varphi - \tilde \varphi\|_{\rho-\delta} + \frac{2M}{\delta} \| \varphi - \tilde \varphi\|_{\rho-\delta} \; \sum_{q=1}^\nu \left(\hskip -1ex \begin{array}{c}  \nu \\ q\end{array} \hskip -1ex 
\right) \sum_{\bf i}
\tilde C_{\bf i}^{q}  \,  \left(\frac{16 M \eps}{\delta}\right)^{|{\bf i}|} \\
& +  \frac{1}{8  \eps}  \sum_{q=1}^\nu \left(\hskip -1ex \begin{array}{c}  \nu \\ q\end{array} \hskip -1ex 
\right)  \sum_{\bf i}
\tilde C_{\bf i}^{q} \,  \left(\frac{16 M \eps }{\delta}\right)^{|{\bf i}|} \sum_{j=1}^q \frac{i_j}{j!} \, \| \partial_\theta^j (\varphi - \tilde \varphi)\|_{\rho-\delta} \\
\leq & \frac{2M}{\delta} P_\nu \left(\frac{16 M  \eps}{\delta}\right)  \| \varphi - \tilde \varphi\|_{\rho,\nu}
\end{align*}
where we denote $\tilde C_{\bf i}^q =  C_{\bf i}^q \; \prod_{j=1}^q j!^{i_j}$ and 
\begin{align*}
P_\nu(\xi) = 1+\sum_{q=1}^\nu \left(\hskip -1ex \begin{array}{c}  \nu \\ q\end{array} \hskip -1ex 
\right)  \sum_{\bf i}
\tilde C_{\bf i}^{q} \,  \left(\xi^{|{\bf i}|} + |{\bf i}| \,  \xi^{|{\bf i}|-1}\right)= \left. \frac{d^\nu} {d x^\nu} \left(\frac{  e^{x+\frac{\xi x}{1-x}}    }    {1-x}                                       \right) \right|_{x=0}.
\end{align*} 
Considering now the second part of $\partial_\theta^\nu (\Lambda(\varphi)_\theta-\Lambda(\tilde \varphi)_\theta)$ and denoting $F=\langle \partial_u \varphi \rangle^{-1} \langle f \circ \varphi \rangle$ and $\tilde F=\langle \partial_u \tilde \varphi \rangle^{-1} \langle f \circ \tilde \varphi \rangle$, we have for $\nu \geq 1$
\begin{align*}
\left\|\partial_\theta^\nu (\partial_u \varphi \, F - \partial_u \tilde \varphi \, \tilde F)\right\|_{\rho-\delta} 
&\leq \left\|\partial_\theta^\nu \partial_u \varphi \right\|_{\rho-\delta} \left\|F-\tilde F \right\|_{\rho-\delta} + \left\|\partial_\theta^\nu \partial_u \varphi - \partial_\theta^\nu \partial_u \tilde \varphi \right\|_{\rho-\delta} \left\|\tilde F \right\|_{\rho-\delta} \\
&\leq \frac{8 M \eps}{\delta} \left\|F-\tilde F \right\|_{\rho-\delta} + \frac{2 M}{\delta} 
\|\partial_\theta^\nu (\varphi- \tilde \varphi)\|_{\rho}
\end{align*}
Denoting $A_\theta(u)=\partial_u \varphi_{\theta}(u)$ and $\tilde A_\theta(u)=\partial_u \tilde \varphi_{\theta}(u)$, we have 
\begin{align*}
\|F-\tilde F\|_{\rho-\delta}
& \leq
\|\langle A\rangle^{-1}\|_{\rho-\delta}   \;
\|\langle f \circ \varphi - f \circ \tilde \varphi\rangle\|_{\rho-\delta} 
+\|\langle A\rangle^{-1}-\langle\tilde A\rangle^{-1}\|_{\rho-\delta}   \;
\|\langle f \circ \tilde \varphi\rangle\|_{\rho-\delta} \\
& \leq \frac{8M}{\delta}  \; \|\varphi - \tilde \varphi\|_{\rho}
\end{align*}
where we have used the bound in (a) 
$\|\langle A\rangle^{-1}\|_{\rho-\delta} \leq 2$, $\|\langle \tilde A\rangle^{-1}\|_{\rho-\delta} \leq 2$  and the inequality 
$$
 \|\langle A\rangle^{-1}-\langle \tilde A\rangle^{-1}\|_{\rho-\delta} \leq \|\langle A\rangle^{-1}\|_{\rho-\delta} \|\langle \tilde A\rangle^{-1}\|_{\rho-\delta} \|A-\tilde A\|_{\rho-\delta} \leq \frac{4}{\delta} \|\varphi - \tilde \varphi\|_\rho.
$$
Finally, we obtain
\begin{align} \label{eq:secondhalf}
\left\|\partial_\theta^\nu (\partial_u \varphi \, F - \partial_u \tilde \varphi \, \tilde F)\right\|_{\rho-\delta} 
&\leq \frac{64 M^2  \eps}{\delta^2}  \; \|\varphi - \tilde \varphi\|_{\rho}+ \frac{2 M}{\delta} 
\|\partial_\theta^\nu \varphi-\partial_\theta^\nu \tilde \varphi\|_{\rho}.
\end{align}
Collecting all terms, we finally have
\begin{align*}
\|\partial_\theta^\nu (\Lambda(\varphi)-\Lambda(\tilde \varphi))\|_{\rho-\delta} \leq 
\frac{2M}{\delta} Q_\nu \left(\frac{16 M \eps}{\delta}\right)  \| \varphi - \tilde \varphi\|_{\rho,\nu}
\end{align*}
with $Q_\nu(\xi) = 1+2 \xi + P_\nu(\xi)$. By a Cauchy estimate it holds that 
$$
\forall \xi  \geq 0, \quad P_\nu(\xi) \leq  2^{\nu+1} \sqrt{e} \, e^{\xi} \, \nu! \quad \mbox{ and } \quad Q_\nu(0) \leq Q_\nu(\xi) \leq  2^{\nu+1} \sqrt{e} \, e^{\xi} \, \nu! + 2 \xi + 1.
$$
The corresponding bound for $\Gamma^\eps$ is then obtained straightforwardly. 
\end{proof}

\bibliographystyle{abbrv}
\bibliography{biblioUA}

\end{document}